\documentclass[11pt]{amsart}

\usepackage[colorlinks=true, pdfstartview=FitV, linkcolor=blue, 
citecolor=blue]{hyperref}
\usepackage[usenames,dvipsnames]{xcolor}

\usepackage{amssymb,amsmath,amscd}
\usepackage{graphicx}
\usepackage{bbm}
\usepackage{a4wide}
\usepackage{enumitem}
\usepackage{verbatim}
\usepackage{float}
\usepackage{accents}
\usepackage{stackrel}
\usepackage{subfig}
\usepackage{faktor}
\usepackage{epsfig,psfrag}  % added psfrag
\usepackage{etex} %fuer gleichzeitigen Gebrauch fuer tikz und pstricks-add
\usepackage{pgfplots}
\usepackage{array}
\usepackage{mathrsfs}
\usepackage{amsthm}
\usepackage{mathtools}
\usepackage{tikz}
\usepackage{wasysym}

\makeatletter
\newcommand*{\boxwedge}{%
  \mathbin{%
    \mathpalette\@boxwedge{}%
  }%
}
\newcommand*{\@boxwedge}[2]{%
  \sbox0{$#1\boxplus\m@th$}%
  \dimen2=.5\dimexpr\wd0-\ht0-\dp0\relax % side bearing
  \dimen@=\dimexpr\ht0+\dp0\relax
  \def\lw{.06}% linw width as factor for height of \boxplus
  \kern\dimen2 % side bearing
  \tikz[
    line width=\lw\dimen@,
    line join=round,
    x=\dimen@,
    y=\dimen@,
  ]
  \draw
    (\lw/2,0) rectangle (1-\lw,1-\lw)
    (\lw,0) -- (.5,1-\lw-\lw/2) -- (1-\lw-\lw/2 ,0)
  ;%
  \kern\dimen2 % side bearing
}
\makeatother

\theoremstyle{plain}
\newtheorem{theorem}{Theorem}[section]

\newtheorem{lemma}[theorem]{Lemma}

\theoremstyle{definition}
\newtheorem{remark}[theorem]{Remark}
\newtheorem{example}[theorem]{Example}
\newtheorem{definition}[theorem]{Definition}

\newcommand{\ts}{\hspace{0.5pt}}
\newcommand{\nts}{\hspace{-0.5pt}}

\newcommand{\RR}{\mathbb{R}\ts}
\newcommand{\PP}{\mathbb{P}\ts}

\newcommand{\ZZ}{{\ts \mathbb{Z}}}

\newcommand{\NN}{\mathbb{N}}

\newcommand{\EE}{\mathbb{E}}
\newcommand{\cA}{\mathcal{A}}
\newcommand{\cB}{\mathcal{B}}

\newcommand{\cN}{\mathcal{N}}

\newcommand{\cO}{\mathcal{O}}

\newcommand{\cR}{\mathcal{R}}
\newcommand{\cS}{\mathcal{S}}
\newcommand{\cT}{\mathcal{T}}

\newcommand{\cX}{\mathcal{X}}

\newcommand{\one}{\mathbbm{1}}

\newcommand{\bP}{{\textbf{P}}}
\newcommand{\fm}{{\mathfrak{m}}}

\newcommand{\tA}{{\texttt{A}}}
\newcommand{\tC}{{\texttt{C}}}
\newcommand{\tG}{{\texttt{G}}}
\newcommand{\tT}{{\texttt{T}}}

\newcommand{\udo}[1]{\underaccent{$\text{.}$}{#1\ts}\nts}

\definecolor{gre}{rgb}{.06,.49,0.03} % hexa 107d08

\DeclareMathOperator{\supp}{supp}

\DeclareMathOperator{\id}{Id}

\newcommand{\defeq}{\mathrel{\mathop:}=}

\tikzset{
    cross/.pic = {
    \draw[rotate = 45] (-#1,0) -- (#1,0);
    \draw[rotate = 45] (0,-#1) -- (0, #1);
    }
}

\begin{document}

\title[Loose linkage in the ARG] {Asymptotic sampling distributions made easy: loose linkage in the ancestral recombination graph}

\author{F. Alberti}
\address[F. Alberti]{Institute for Mathematics,  Johannes Gutenberg University Mainz, \newline
\hspace*{\parindent}55128 Mainz, Germany}
\email{fralbert@uni-mainz.de}

\begin{abstract}
Understanding the interplay between recombination and genetic drift is a significant challenge in mathematical population genetics and of great practical interest.
Asymptotic results about the distribution of samples when recombination is strong compared to genetic drift are often based on the approximate solution of certain recursions, which is technically hard and offers little conceptual insight.
This work generalises an elegant probabilistic argument, based on the coupling of ancestral processes but so far only available in the case of two sites, to the multilocus setting. This offers an alternative route to, and slightly generalises, a classical result of Bhaskar and Song.
\end{abstract}

\maketitle


\bigskip


\section{Introduction}
Recombination is a genetic mechanism that reshuffles genetic information between parent individuals in the context of sexual reproduction, and plays an important role in the creation and maintenance of genetic diversity. Due to their inherent nonlinearity, mathematical models for recombination have been a major challenge for mathematical population geneticists since their conception more than one hundred years ago~\cite{Jennings1917,Robbins1918}.  

A powerful idea in modern mathematical population genetics is to understand the effect of evolution on the distribution of a sample by considering its evolutionary history backward in time by means of \emph{ancestral processes}. In particular, the \emph{ancestral recombination graph} (ARG)~\cite{Hudson83,GriffithsMarjoram96,GriffithsMarjoram97,JenkinsFearnheadSong2015} describes the coupled ancestries of multiple genetic sites under recombination. 

When neglecting the effect of random genetic drift\footnote{In population genetics, \emph{genetic drift} refers to random fluctuations in type frequencies due to variability in family sizes. The reader should not confuse this with the drift term of a diffusion.}, the evolution of the genetic composition forward in time can be described by a system of coupled, nonlinear ordinary differential equations~\cite{Buerger2000}.
Backward in time, the genealogy of a sample is described by an ARG which is still random, but simplifies considerably in this setting;
genetic drift corresponds to coalescence of lineages in the ARG and its absence therefore implies their conditional independence. Indeed, it was shown that the ARG reduces to a simple \emph{partitioning process}~\cite{haldane} that describes the distribution of an individual's genetic heritage across its ancestors. This led to a linearisation of the nonlinear forward-time model and an explicit solution. Later, this approach was extended to include other evolutionary forces such as selection~\cite{AlbertiBaake21, AlbertiBaakeHerrmann21}, 
migration~\cite{AlbertiBaakeLetterMartinez21} as well as mutation and more~\cite{Alberti22p}. See also~\cite{recoreview} for a very readable survey.

In the majority of natural populations, however, genetic drift plays an important role and thus, coalescence is an integral part when modelling sample genealogies via the ARG. Unfortunately, explicit, closed formulae for sampling probabilities are not known in this case. Thus, much effort has been put into the development of efficient computational algorithms; as of today, there exist pre-built software packages like msprime that allow for the efficient simulation of the ARG even for a large number of loci~\cite{EtheridgeKelleherMcVean2016,KelleherLohse2020,Baumdickeretal2022}. However, the state space of potential genealogies is so large that parameter inference still remains computationally expensive. Simulating the ARG conditional on the configuration of a sample is similarly expensive. Finding efficient computational solutions to these two problems remains an active area of investigation to this day. See also~\cite{GriffithsMarjoram96, FearnheadDonnelly2001, GriffithsJenkinsSong2008, JenkinsGriffiths2011,KuhnerYamatoFelsenstein2000,Nielsen2000,WangRannala2008, Rasmussen2014,BoitardLoisel2007,Miura2011} for background.

At the same time, \emph{approximate} analytic formulae and expansions for the sampling distribution are available when recombination is strong compared to 
genetic drift~\cite{JenkinsSong2009,JenkinsSong2010,BhaskarSong2012,JenkinsSong2012,JenkinsFearnheadSong2015}.
Much of this work is based on a recursive representation for the sampling distribution that  goes  back to Golding~\cite{Golding1984}; 
see also~\cite{EthierGriffiths1990}. In particular, this includes the work~\cite{BhaskarSong2012} of Bhaskar and Song, which seems to be the only work to date that
treats more than two sites. It is based on a generalisation of Golding's recursion to more than two sites. The authors show that a certain ansatz satisfies Golding's recursion up to a certain order. This is technically challenging and, more importantly, it is not easy to see where the ansatz comes from. While the authors give some probabilistic interpretation in terms of Wallenius' noncentral hypergeometric distributions, there is no clear connection to the underlying ancestral structure.

Therefore, it seems desirable to find an alternative approach that exploits the genealogical structure of the problem more efficiently. 
This was accomplished in~\cite{JenkinsFearnheadSong2015} for the case of two sites. Based on the insight that for `infinitely' strong recombination, the ARG reduces to a collection of independent Kingman coalescents, one for each site, the authors constructed a coupling between such a collection and the full ARG.
This coupling allowed for an easy derivation of the first-order correction in the reciprocal of the recombination rate. This correction can, in, fact, be expressed in terms of single-site sampling distributions, in a way that does not depend on the mutation model used; this property was called \emph{universality} by the authors.  

The main goal of this work is to provide an alternative proof of the general result of Bhaskar and Song~\cite{BhaskarSong2012} that yields deeper insight into 
the probabilistic structure of the problem. Incidentally, this allows us to generalise their result from single-crossover recombination to arbitrary recombination patterns without additional effort. While this level of generality may not seem particularly relevant in view of applications where one is mostly concerned with single-crossover, it serves to showcase the power and flexibility of ancestral models for recombination, in the tradition 
of~\cite{bb,haldane,recoreview}.

Inspired by the argument in~\cite{JenkinsFearnheadSong2015}, which was based on an \emph{untyped} version of the ARG, the main difference is that we will rely on a \emph{typed} version of the ARG; an important tool is its efficient description as a Markov process with values in the finite point measures on the type space. Such typed coalescents have been investigated for population evolving under general genic selection~\cite{EtheridgeGriffiths2009}, also in the large-population limit with asymptotically infinite offspring variance~\cite{EtheridgeGriffithsTaylor2010}. For a nice monograph on measure-valued population processes and their duals, see~\cite{DawsonGreven2014}. The major novelty of our approach is that rather than assigning fixed types to the ancestral lineages, we introduce a notion of \emph{fuzzy types} that will allow us to track the multitude of possible ancestral types that lead to a given sample.

The rest of the paper is organised as follows. To make the presentation as self-contained as possible, we start by recalling the ARG and the associated \emph{sampling-distribution}. In Section \ref{sec:marg}, we describe the typed, measure-valued ARG, along with a version for infinitely strong recombination.
In Section~\ref{sec:coupling}, we couple these two processes and derive the first-order correction in terms of the single-site sampling distributions and probabilities of certain events in this coupling. These probabilities are evaluated in Section~\ref{sec:taming} via an elementary inclusion-exclusion argument, which ultimately leads to the formula given in~\cite{BhaskarSong2012}.

\section{The ancestral recombination graph and the distribution of a sample}\label{sec:informalarg}

\subsection{Basic notions and notations} \label{subsec:notation}
We are interested in the genotype composition of a finite random sample from a stationary population that has evolved under recombination, 
site-independent mutation and genetic drift. We assume that it consists of infinitely many haploid individuals, where \emph{haploid} means that each individual carries  a single set of genes only. In particular, genotypes are identified with genetic sequences. We denote the set of genetic sites by
\begin{equation*}
S = \{1,\ldots,n\},
\end{equation*}
and the set of potential alleles at site $i \in S$ by $X_i^{}$. We will assume the $X_i^{}$ to be finite. Generic alleles at site $i$ will be denoted  by $x_i^{},y_i^{},z_i^{},\ldots$. 
With this, 
\begin{equation*}
X \defeq \prod_{i \in S} X_i
\end{equation*}
is the set of \emph{complete genotypes}. 
In order to describe samples in which individuals may be observed at different  subsets of sites, 
we also define for any $A \subseteq S$ the set
\begin{equation*}
X_A^{} \defeq \prod_{i \in A} X_i^{}
\end{equation*}
of  \emph{partial genotypes observed at $A$}; for $A = \varnothing$, we set $X_\varnothing \defeq \{\epsilon\}$, where $\epsilon$ is the empty sequence. The set of \emph{partial genotypes} is then given by the disjoint union
\begin{equation*}
\cX \defeq \dot{ \bigcup_{A \subseteq S}} X_A^{}.
\end{equation*}
For simplicity, we will refer to the elements of $\cX$ as \emph{types} and denote, for any $x \in \cX$, by $D(x)$ the unique $A \subseteq S$ with
$x \in X_A^{}$. We say that an individual of type $x$ is \emph{observed at $D(x)$}.

It is important to bear in mind that $X_A^{} \cap X_B^{} = \varnothing$ if $A \neq B$. This is true even if $A$ and $B$ are singletons and all the alphabets $X_i$ agree; for this, we need to make (for any $i \in S$) the distinction between the set $X_{ \{ i    \} }$ of types observed at site $i$ and the set $X_i$ of potential alleles at site $i$. Formally, the former may be viewed as the set of mappings from $\{ i \}$ to $X_i$. 
\begin{example}
Here and in all later examples, we will write types as finite words, where unobserved sites are marked as $\ast$. So, for $n = 3$ and, say,
$X_1 = X_2 = X_3 = \{\texttt{A}, \texttt{C}, \texttt{G}, \texttt{T} \}$, we have 
$X_{\{ 1 \}} = \{ \texttt{A}\ast \ast , \texttt{C} \ast \ast , \texttt{G} \ast \ast, \texttt{T} \ast \ast    \}$, which is clearly distinct from
$X_{\{ 2 \}} = \{ \ast \texttt{A} \ast, \ast \texttt{C} \ast, \ast \texttt{G} \ast,  \ast \texttt{T} \ast \}$.
\end{example}

Type frequencies within samples are represented as finite counting measures on $\cX$. Generally, we denote the set of finite counting measures on a set $M$ by
\begin{equation*}
\cN(M) \defeq \Big \{ \sum_{x \in \cT} \nu_x^{} \delta_x^{} : \cT \subseteq M \textnormal{ finite}, \nu_x^{} \in \NN \Big \},
\end{equation*}
where $\delta_x^{}$ is the point mass (or Dirac-measure) in $x$ and $\NN$ does not include $0$. The set $\cT \subseteq M$ is called the \emph{support} ($\supp(\nu)$) of 
$\nu = \sum_{x \in \cT} \nu_x^{} \delta_x^{}$, and the total mass of $\nu$ is denoted by $\|\nu\| \defeq \sum_{x \in \cT} \nu_x^{}$.
%The unit point masses are also referred to as \emph{particles}.

We need some additional notation regarding types and type compositions. For $x \in \cX$ and $i \in D(x)$, we write $x(i)$ for the  allele at site $i$, i.e. $x = (x(i))_{i \in D(x)}$. Moreover, given $x \in \cX$ and $B \subseteq S$, we write $x|_{B}^{}$ for the \emph{marginal of $x$ with respect to $B$}, that is, the type
$y \in X_{D(x) \cap B}^{}$ with $y(i)^{} = x(i)^{}$ for all 
$i \in D(x) \cap B$. We will also write $x|_i^{}$ instead of  $x|_{\{i\}}^{}$ to save brackets. Note again the subtle difference between $x(i)$ and $x|_i$.
The former is the allele at site $i$, that is, an element of $X_i$. The latter is a type observed at site $i$, i.e., an element of $X_{\{i\}}$.

We call two types $x,y \in \cX$ \emph{compatible} 
if $x|_{D(x) \cap D(y)}^{} = y|_{D(x) \cap D(y)}^{}$, and \emph{incompatible} otherwise. 
Note that $x$ and $y$ are always compatible if 
$D(x) \cap D(y) = \varnothing$.
For compatible types, we define their \emph{join} $x \sqcup y \in X_{D(x) \cup D(y)}^{}$ via
$(x \sqcup y)(i) \defeq x(i)$ for $i \in D(x)$ and $(x \sqcup y)(i) \defeq y(i)$ for $i \in D(y)$; of course, the join is not well-defined if $x$ and $y$ are incompatible.

For $\nu \in \cN(\cX)$ and $A \subseteq S$, we write $\nu^{\supseteq A}$ for the restriction of $\nu$ to $\bigcup_{B \supseteq A} X_B$. Given 
$B \subseteq S$, we write $\nu^{B}$ for the marginal of $\nu$ w.r.t. $B$, defined for all $E \subseteq X_B^{}$ via
\begin{equation*}
\nu^{B} (E) \defeq \sum_{\substack{x \in \cX \\ x|_B^{} \in E}} \nu(x).
\end{equation*}
Finally, we write $\nu^{\supseteq A,B}$ for the marginal of $\nu^{\supseteq A}$ with respect to $B$. In words, $\nu^{\supseteq A}$ describes the subsample consisting of individuals that are observed at a superset of $A$, while $\nu^{B}(x_B^{})$ for $x_B^{} \in X_B$ is the number of individuals that have the marginal type $x_B^{}$ with respect to $B$. Note that $\nu^{\supseteq B, B} = \nu^B$ and that $\nu^\varnothing (\epsilon)$ is the total mass $\|\nu\|$ of 
$\nu$. Let us consider an example to illustrate this notation.
\begin{example} \label{ex:marginalcomputation}
Let $n = 3$ and hence $S = \{1,2,3\}$, $X_1 = X_2 = X_3 =  \{\tA,\tC,\tG,\tT \}$. 
Consider $\nu = 2 \delta_{\tA \tG \tC} + \delta_{\ast \tC \tA} + \delta_{\tA \ast \tA} + \delta_{\tC \tG \ast} + \delta_{\tA \ast \ast}$.
Then, for $A = \{1,2\}$, $B = \{2\}$,
\begin{equation*}
\begin{split}
\nu^{\supseteq A} &= 2 \delta_{\tA \tG \tC} + \delta_{\tC \tG \ast}, \\
\nu^B &= 3 \delta_{\ast \tG \ast} + \delta_{\ast \tC \ast}, \\
\nu^{\supseteq A, B}  &= 3 \delta_{\ast \tG \ast}.
\end{split}
\end{equation*}
 
\end{example}

To describe the effect of recombination, we use partitions of $S$ to describe the fragmentation of genetic material of individuals across their parents (this usage is as in~\cite{haldane}). Recall that a \emph{partition of a set $M$} is a collection of nonempty, mutually disjoint subsets of $M$, called \emph{blocks}, that cover $M$. The set of all partitions of $S$ is denoted by $\bP(S)$; for any $\cA \in \bP(S)$ and for any nonempty $B \subseteq S$, we call
\begin{equation*}
\cA|_B^{} \defeq \{ A \cap B : A \in \cA \} \setminus \{ \varnothing \}
\end{equation*}
the partition \emph{induced by $\cA$ on $B$}.

We assume that mutation occurs independently at different sites and with rate $u_i^{} \geqslant 0$ at site $i$. Upon mutation at site $i$, and for generic alleles
$x_i^{}, y_i^{} \in X_i$, we write 
$M_i^{} (x_i^{},y_i^{})$ for the probability that an allele $x_i^{} \in X_i^{}$ is replaced by $y_i^{} \in X_i^{}$. We call $M_i^{}$ the \emph{mutation kernel at site $i$}; as the $X^{}_i$ are finite, $M_i$ can be interpreted as a Markov matrix and $u_i^{} (M_i - \id)$ is a generator. We call the associated continuous-time Markov chain on $X_i^{}$ the \emph{mutation chain} at site $i$. In order to guarantee asymptotically stable allele frequencies, we will assume that the $M_i^{}$ are irreducible. This implies that the mutation chain at any site $i$ converges to its unique stationary distribution $\pi_i^{}$, given by the unique left eigenvector of $M_i^{}$ with respect to the eigenvalue $1$. 

We close by summarising our notation for easy reference. The set of types is denoted by $\cX$, generic types as $x,y,z,\ldots$. The set over which a type $x$ is observed is $D(x)$, and for each $i \in D(x)$, $x(i)$ is the allele carried by an individual of type at site $i$. Marginal types are written as 
$x|_A^{}$ and marginal types with respect to singletons $x|_i^{}$ must be distinguished from the allele $x(i)$ at site a in a type $x$, and from generic elements $x_i^{},y_i^{},z_i^{},\ldots$ of $X_i$. Whenever we want to indicate singleton types that are observed at a particular site $i$, we will write $x^i,y^i,z^i,\ldots$.

\subsection{The ancestral recombination graph} \label{subsec:arg}

To understand how the composition of a sample is affected by recombination, mutation and genetic drift, we will consider its evolutionary history,  which is captured by the \emph{ancestral recombination graph} (ARG); see~\cite{Hudson83,GriffithsMarjoram96,GriffithsMarjoram97,JenkinsFearnheadSong2015}. 
We give a somewhat informal graphical description.

A sample from the population at present is represented by a finite set of leaves and their ancestral lines as vertical lines, starting  from the leaves and growing from bottom to top. We call the ancestral line associated with a leaf observed at $A \subseteq S$ \emph{ancestral to $A$}. Somewhat abusively, we say that a line is ancestral to a site $i$ if it is ancestral to $A$ with $i \in A$. If we want to express that a line is ancestral to site $i$ and only to site $i$, we will say that
it is ancestral to $\{ i \}$.

 Recombination, mutation and genetic drift manifest themselves as follows (See Fig.~\ref{fig:arg} for an illustration).
\begin{enumerate}
\item
The effect of genetic drift is captured by the pairwise coalescence of lines, indicating that a pair of individuals has found a common ancestor. Any ordered pair of ancestral lines, say, ancestral to $A$ and $B$, coalesces, independently of all the other pairs, at rate $1$ into a single line ancestral to $A \cup B$. 
\item
Recombination results in the \emph{fragmentation} of lineages. Independently for any partition $\cA$ of $S$, any line is independently of all the others and at rate 
$\varrho_\cA^{} \geqslant 0$ hit by an $\cA$-recombination event 
\begin{tikzpicture}[baseline=15.75mm]
\draw (0.1,1.5) +(-0.2,0.4) rectangle +(0.2,0);
\node at (0.1,1.7) {$\scriptstyle{\cA}$};
\end{tikzpicture},
indicating that the individual at the affected (\emph{offspring}) line is the offspring of $|\cA|$ parents, corresponding to the blocks of $\cA$; the elements of each block are the sites that the offspring inherits from the corresponding parent. This leads to the fragmentation of the offspring line into multiple parental lines. If the offspring line is ancestral to $B \subseteq S$, the parental lines are ancestral to the blocks of the induced partition $\cA|_B^{}$ rather than $\cA$, because only the sites in $B$ are relevant for the sample. In particular, the transition is silent whenever $|\cA|_B| = 1$, that is, whenever $B \subseteq A$ for some $A \in \cA$.
%fragmented into multiple \emph{parental} lines which are ancestral to the blocks of $\cA|_B^{}$.
%This indicates that the individual at the affected line (which we also call the \emph{offspring line}) is the offspring of an arbitrary number of parents, %corresponding to the blocks of $\cA|_B^{}$, represented by the various parental lines. Here, the elements of each block are the sites that are contributed by %the corresponding parent. The rates $\varrho_\cA^{} \geqslant 0$ are referred to as \emph{recombination rates}.
\item
Mutations are represented by \emph{mutation marks}  
\begin{tikzpicture}[baseline=-1.4mm]
\draw (0,0) circle (6pt);
\node at (0,0) {$i$};
\node[anchor=west] at (0.1,-0.1) {$x_i^{} \to y_i^{}$};
\end{tikzpicture}
which indicate mutation at site $i$ from the allele $x_i^{}$ to the allele $y_i^{}$. They appear on each line independently with respective rates $u_i^{} M_i(x_i^{},y_i^{})$, independently for each $i \in S$ and $x_i^{}, y_i^{} \in X_i^{}$. 
\end{enumerate}
See Fig.~\ref{fig:arg} for an illustration. 
\begin{remark}
As there is (at the time of its construction) no fixed initial type configuration in the ARG, we have to account in 3 for mutations starting from \emph{all} possible alleles. Thus, most mutations will usually be silent, such as the mutation on the rightmost black line in Fig.~\ref{fig:arg}.
\end{remark}
\begin{remark} \label{rem:kingman}
An important special case arises when each leaf is ancestral to the same singleton, say $\{i\}$, recombination events 2 are void and the ARG reduces to a simple coalescent process in which each (ordered) pair of lineages coalesces at rate $1$ and which is decorated by mutation events. This process is known as \emph{Kingman's coalescent}. In Section~\ref{sec:coupling}, we will explain how to use independent Kingman coalescents to approximate the ARG in the setting of strong recombination. 
\end{remark}

\begin{figure}[t]
\centering
\begin{tikzpicture}
\draw[very thick,color=red] (0,0) -- (0,1.5);
\draw[very thick,color=blue] (0.1,0) -- (0.1,1.5);
\draw[very thick] (0.2,0) -- (0.2,1.5);
\draw (0.1,1.5) +(-0.2,0.4) rectangle +(0.2,0);
\node at (0.1,1.7) {$\scriptstyle{\cA}$};
\draw[very thick, color=red] (-0.1,1.7) -- (-0.9,1.7) -- (-0.9,6);
\node[anchor= north] at (0.1,0) {$\scriptstyle{(\texttt G, \texttt C, \texttt A)}$};
\draw[very thick, color=blue]  (0.3,1.75) -- (0.8,1.75) -- (0.8,3);
\draw[very thick, color=black] (0.3,1.65) -- (0.9,1.65) -- (0.9,3);
\draw (0.85,3) +(-0.2,0.4) rectangle +(0.2,0);
\node at (0.85,3.2) {$\scriptstyle{\cB}$};
\draw[very thick, color=blue] (0.65,3.2) -- (0.45,3.2) -- (0.45,5) -- (0.8,5) -- (0.8,6);
\draw[very thick, color=black] (1.05,3.2) -- (1.25,3.2) -- (1.25,5) -- (0.9,5) --(0.9,6); 
\draw[very thick,color=blue] (3,0) -- (3,2) -- (0.8,2);
\draw[very thick,color=black] (3.1,0) -- (3.1,2.1) -- (0.9,2.1);
\node[anchor=north] at (3,0) {$\scriptstyle{(\ast, \texttt T, \texttt A)}$};
\draw[very thick,color=black] (6,0) -- (6,5.5) -- (0.9,5.5);
\node[anchor=north] at (6,0) {$\scriptstyle{(\ast, \ast, \texttt A)}$};
\draw[-latex] (7,0) -- (7,6);
\draw[dashed] (6,5.5) -- (7,5.5);
\node[anchor=west] at (7,5.5) {$T_{\textnormal{MRCA}}^{}$};
%%%TYPES%%%
\node[anchor=south] at (-0.9,6) {$\scriptstyle{(\texttt G,\ast,\ast)}$};
\node[anchor=south] at (0.85,6) {$\scriptstyle{(\ast,\texttt C, \texttt A)}$};
\node[anchor=south] at (2.5,2.1) {$\scriptstyle{(\ast,\texttt C, \texttt A)}$};
\node[anchor=south] at (4.5,5.5) {$\scriptstyle{(\ast,\ast,\texttt A)}$};
\node[anchor=east] at (0.45,4.3) {$\scriptstyle{(\ast,\texttt C, \ast)}$};
\node[anchor=west] at (1.25,4.3) {$\scriptstyle{(\ast,\ast,\texttt A)}$};
\node[anchor=west] at (0.9,2.7) {$\scriptstyle{(\ast, \texttt C, \texttt A)}$};
\node[anchor=north west] at (0.8,1.7) {$\scriptstyle{(\ast,\texttt C, \texttt A)}$};
%%MUTATIONS%%%
\draw[fill=white] (3.05,1.3) circle (6pt);
\node at (3.05,1.3) {$\scriptstyle{2}$};
\node[anchor=west] at (3.16,1.3) {$\scriptstyle{\texttt C \to \texttt T}$};
\draw[fill=white] (6,4) circle (6pt);
\node at (6,4) {$\scriptstyle{3}$};
\node[anchor=east] at (5.9,4) {$\scriptstyle{\texttt T \to \texttt C}$};
%%%TIME AXIS LABEL%%%
\node[rotate=90, anchor=west] at (7.2,2) {time};
\end{tikzpicture}
\caption{\label{fig:arg}
A realisation of the ARG, started from $3$ leaves, observed at $\{1,2,3\}$, $\{2,3\}$ and $\{3\}$ (left to right). The first, second and third site are encoded red, blue and black and the sets of alleles are given by $X_1^{} = X_2^{} = X_3^{} = \{\texttt{A},\texttt{C},\texttt{G},\texttt{T} \}$. Unobserved sites are denoted by $\ast$. The fragmentation events are governed by the partitions $\cA = \{ \{1\}, \{2,3\} \}$ and $\cB = \{ \{1,3\}, \{2\} \}$.
Note that we have two common ancestors at time $T_{\textnormal{MRCA}}^{}$; a red one, ancestral to $\{1\}$, and a blue-black one, 
ancestral to $\{2,3\}$. The letters $\texttt G$, $\texttt C$ and $\texttt A$ in the ancestral sequences are the result of independent mutation beyond $T_{\textnormal{MRCA}}^{}$ on the red, blue and black lines, respectively.}

\end{figure}
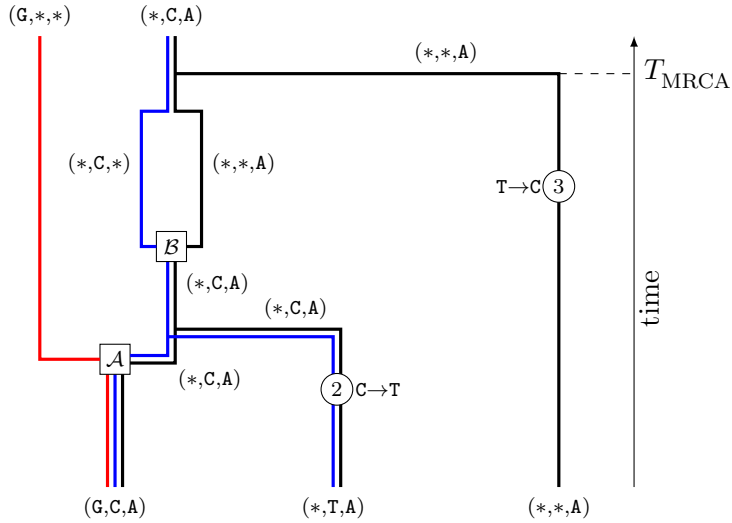

The ARG enables us to relate the type configuration of a sample of individuals from the present to that of their ancestors at any previous time. Given a realisation of the ARG with finite runtime, we call the top ends of the lines the \emph{ancestors} of the sample. More specifically, the top end of a line ancestral to $A$ corresponds to an ancestor observed at $A$. Given an assignment of types to the ancestors (where an ancestor observed at $A$ is assigned a type in $X_A^{}$), we propagate the types along the lines, from top to bottom, until we arrive at the leaves. We proceed according to the following rules.
\begin{enumerate}
\item
When encountering a recombination event, the types $x^{(1)}, \ldots, x^{(r)}$ at the parental lines are joined to form the type
$x^{(1)} \sqcup \ldots \sqcup x^{(r)}$ at the offspring line. Note that the fact that blocks of a partition are pairwise disjoint implies that
$D(x^{(1)}),\ldots,D(x^{(r)})$ are pairwise disjoint (As can be seen, for instance, in Fig.~\ref{fig:arg}). Hence, $x^{(1)},\ldots,x^{(r)}$ are pairwise compatible and the join
$x^{(1)} \sqcup \ldots \sqcup x^{(r)}$ is well defined.
\item
When encountering a coalescence event (say, between lines ancestral to $A$ and $B$), if $x$ is the type of the common ancestor, then $x|_C^{}$ is the type at the line that participated in the coalescence that is ancestral to $C \in \{A,B\}$.  
\item
When a mutation 
\begin{tikzpicture}[baseline=-1.4mm]
\draw (0,0) circle (6pt);
\node at (0,0) {$i$};
\node[anchor=west] at (0.1,-0.1) {$x_i^{} \to y_i^{}$};
\end{tikzpicture}
is encountered while the allele at site $i$ at that line is $x_i^{}$, it is changed to $y_i^{}$; otherwise, the mutation is silent.
\end{enumerate}
\begin{remark}
Note that rule 3 together with the rates at which mutation marks appear indeed makes it so that the alleles at different sites for each individual effectively evolve according to the independent mutation chains.
\end{remark}

\subsection{Approximate sampling formula} \label{subsec:sampling}

We want to use the ARG to understand the distribution of a sample drawn at \emph{stationarity}. Recall that in order to guarantee asymptotically stable allele frequencies, we assumed that the mutation kernels $M_i^{}$ are irreducible. We will now argue that this is indeed sufficient for the configuration of samples consisting of multiple individuals to have a unique stationary distribution as well.

Let $T_{\textnormal{MRCA}}^{}$ be the \emph{time of the most recent common ancestors}, that is, the first time at which, for each $i \in S$, there is at most one line ancestral to $i$. Notice the plural; classically, in single-site coalescent models, there is a single individual, having lived some time in the past, that is ancestral to the entire sample. Under recombination, however, different sites may have different common ancestors, as in Fig.~\ref{fig:arg}.
 It is clear that  
$T_{\textnormal{MRCA}}^{}$
is almost surely finite.
The types of the ancestors at time $T_{\textnormal{MRCA}}^{}$ are the result of independently evolving mutation chains, run for time 
$t - T_{\textnormal{MRCA}}^{}$, where $t$ is the total runtime of the ARG; this is because every mutation after $T_{\textnormal{MRCA}}^{}$ affects only a single individual alive at time
$T_{\textnormal{MRCA}}^{}$. In the limit $t \to \infty$, the type of an ancestor at time $T_{\textnormal{MRCA}}^{}$ observed at $A$ is distributed according to 
$\bigotimes_{i \in A} \pi_i^{}$, where $\pi_i^{}$ is the stationary distribution of the mutation chain at site $i$. 
To summarise, we can produce a random sample from the population at stationarity in three steps; see also Fig.~\ref{fig:arg}.
\begin{enumerate}
\item
Run the ARG until (backward) time $T_{\textnormal{MRCA}}$. 
\item
Assign to each ancestor (the top end of each ancestral line), observed at, say, $A$, an independent sample from
$
\bigotimes_{i \in A} \pi_i^{}.
$
\item
Propagate these types from top to bottom, following the rules above.
\end{enumerate}

Given $\nu \in \cN(\cX)$, we write $q(\nu)$ for the probability of obtaining an \emph{ordered} sample configuration with frequencies $\nu$, i.e. a sample in which each type $x \in \cX$ appears $\nu(x)$ times and in a particular, fixed (but arbitrary) order when the ARG in step 1 is started from $\nu(X_A^{})$ leaves observed at $A$ for all $A \subseteq S$; we will see a more formal definition in Section~\ref{sec:marg}. Due to exchangeability, that is, the invariance of the evolution of the ARG under permutation of lines, the probability of observing any sample configuration depends only on the type frequencies.
Moreover, it will turn out to be convenient to also define $q_\infty^{} (\nu) \defeq 0$ if $\nu$ is a signed point measure that contains negative masses.

Because we want to investigate a setting of strong recombination, we assume that the recombination rates $\varrho_\cA^{}$ are of the form $\varrho \cdot r_\cA^{}$ where $\varrho$ is a global scaling parameter that we will later send to $\infty$, and $r_\cA^{} \geqslant 0$ are constants. We will further assume without loss of generality that any two sites may be separated by recombination. This means that for any two $i,j \in S$, $i \neq j$, there is an $\cA \in \bP(S)$ with
$\varrho_\cA^{} > 0$ and $\cA|_{\{i,j\}}^{} = \{\{i\},\{j\} \}$. Otherwise we could lump $i$ and $j$ together and treat them as a single site. 
We use the normalisation $\sum_{\cA \in \bP(S)} r_\cA^{}  = 1$, call $\varrho$ the recombination \emph{rate} and refer to the $r_\cA^{}$ as recombination
\emph{probabilities}.

%explain the approximation of the ARG by a collection of independent Kingman coalescents HERE
The ARG is a fairly complicated object and, in particular, explicit formulas for the associated sampling probabilities $q(\nu)$ are not available. Here, we 
%want to use the ARG to (re)derive a formula for $q_1^{}$ in the approximation
consider an asymptotic expansion of the form
\begin{equation} \label{refinedapproximation}
q(\nu) = q_\infty^{} (\nu) + \varrho^{-1} q_1^{}  (\nu) + O(\varrho^{-2}) \quad \textnormal{ as } \varrho \to \infty.
\end{equation}
In order to understand the terms $q_\infty^{} (\nu)$ and $q_1^{} (\nu)$, we will approximate the ARG by a collection 
of independent Kingman coalescents, one for each site, by trying to match transitions of the two processes. In this context, $q_\infty^{} (\nu)$ is the sampling probability associated with such an independent collection of Kingman coalescents and the first-order correction $q_1^{} (\nu)$ accounts for the most likely of those transitions that can not be matched.

To be a bit more precise, let us 
consider for each $i \in S$ the \emph{marginal coalescent associated with site $i$} ($\textnormal{MC}_i$ for short) by tracing in the ARG, starting from each particle observed at site $i$, the unique lineages ancestral to site $i$. For instance, in Fig.~\ref{fig:arg} one would obtain the $\textnormal{MC}_1$ 
($\textnormal{MC}_2$,$\textnormal{MC}_3$) by considering only the red (blue, black) lines and erasing all the others.

Clearly, each $\textnormal{MC}_i$ evolves like a Kingman coalescent (see Remark~\ref{rem:kingman}) and in particular, the law of $\textnormal{MC}_i$ does not depend on $\varrho$. As is evident from Fig.~\ref{fig:arg}, the collection 
$\textnormal{MC}_1, \ldots, \textnormal{MC}_n$ is not independent; however, if $\varrho$ is very large, the ARG consists mostly of lines that are ancestral to singletons. It is not hard to see that this implies that coalescences occur almost independently in different MCs. 

Consequently, for large $\varrho$, the type configuration generated by the ARG can be thought of as pieced together from the samples generated by the approximately independent marginal coalescents $\textnormal{MC}_1,\ldots,\textnormal{MC}_n$. We therefore expect in~\eqref{refinedapproximation} that
\begin{equation} \label{qinfty}
q_\infty^{} (\nu) = \prod_{i=1}^n q (\nu^{\{i\}}),
\end{equation}
which in particular means that $q_\infty^{} (\nu) = q_\infty^{} (\sigma(\nu))$, where $\sigma$ is the \emph{fragmentation map} defined via
\begin{equation} \label{sigmadef}
\sigma(\nu) \defeq \nu^{\{1 \}} + \ldots + \nu^{ \{n \}}.
\end{equation}
By the linearity of marginalisation, it is clear that $\sigma$ is linear.
In order to derive an expression for $q_1^{}$ in Eq.~\eqref{refinedapproximation}, we will later (in Section~\ref{sec:coupling}) construct an explicit coupling of the ARG and a collection of a-priori independent Kingman coalescents, one for each site. It turns out that this can be done in such a way that the latter match up precisely with the marginal coalescents of the ARG with high probability. This will immediately prove that $q(\nu) = q_\infty^{} (\nu) + O(\varrho^{-1})$ with $q_\infty^{} (\nu)$ as in Eq.~\eqref{qinfty}. By exploring the most likely ways in which this matching can break down, we obtain the first-order correction in Eq.~\eqref{refinedapproximation}.

In Section~\ref{sec:marg}, we will formally define $q$ and $q_\infty^{}$ via a measure-valued representation of the ARG and a collection of independent Kingman coalescents. There, we will also prove Eq.~\eqref{qinfty}. For now, the reader may think of $q_\infty^{}$ as being defined via Eq.~\eqref{qinfty}.

In the setting of parent-independent mutation, that is, if the mutation probabilities $M_i(x_i^{}, y_i^{})$ can be written as $M_i (y_i^{})$ independently of the parental type $x_i^{}$, 
we have the explicit formula
\begin{equation*}
q (\nu^{\{i\}}) = \frac{1}{(u_i^{})_{\overline{n_i^{}}}}\prod_{y_i^{} \in X_i^{}} \big (u_i^{} M_i^{} (y_i^{})\big)_{\overline{\nu^{\{i\}}(y_i^{})}},
\end{equation*}
where $(z)_{\overline m}^{} \defeq z ( z + 1) \ldots (z + m -1)$ denotes the $m$-th ascending factorial of $z$ and \mbox{$n_i^{} \defeq \|\nu^{\{i\}} \|$} is the number of individuals that are observed at a (possibly proper) superset of $\{i\}$;
see~\cite[Remark 2.2]{JenkinsFearnheadSong2015}. For general finite-allele models, approximate single-site sampling formulae are available in~\cite{BhaskarKammSong2012}. 

The following theorem was proved in~\cite{BhaskarSong2012} for the special case of single-crossover recombination.
\begin{theorem}\label{thm:maintheorem}
The first-order term in the expansion~\eqref{refinedapproximation} is given by
\begin{equation} \label{samplingformula}
q_1^{} (\nu) = \sum_{x \in \cX} q_\infty^{} \big ( \sigma(\nu) - \sigma(\delta_x^{})   \big ) \sum_{ \substack {D(x) \subseteq A \subseteq S\\ |A| \geqslant 2}} 
\frac{(-1)^{|A \setminus D(x)|}}{\bar r_A^{}} \binom{\nu^{\supseteq A,D(x)}(x)}{2},
\end{equation}
where 
\begin{equation*}
\bar r_A^{} \defeq \sum_{\substack{\cB \in \bP(S) \\ \cB|_{A}^{} \neq \{A\}}} r_\cB^{}
\end{equation*}
is the total fragmentation rate of $A$ and $\sigma$ is the linear fragmentation map defined in Eq.~\eqref{sigmadef}. Recall that $q(\mu)$ and therefore
$q_\infty^{} (\mu)$ vanish if $\mu$ is a signed point measure with $\mu(x) < 0$ for some $x \in \cX$. 
\end{theorem}
We stress that the outer sum in \eqref{samplingformula} runs over \emph{all} types, including the empty type $\epsilon$ in which case the inner sum runs over all subsets $A$ of $S$ with at least two elements.  Recall that $\nu^{\supseteq A, D(\epsilon)} (\epsilon)$ is just the total mass $\| \nu^{\supseteq A} \|$ of the measure $\nu^{\supseteq A}$, i.e. the total number of samples that are observed at some superset of $A$. 

Keep in mind that so far, we have defined $q_\infty^{} (\nu)$ only informally as the sampling probability associated with an independent collection of Kingman coalescents. We will give a more precise definition in Section~\ref{sec:marg}.

\begin{example} \label{ex:concreteexample}
As the expression in Eq.~\eqref{samplingformula} looks fairly complex, let us consider an example and compute $q_1^{} (\nu)$ for 
$\nu = 2 \delta_{\texttt{CA}}$.

First of all, we have $\sigma(\nu) = 2 \delta_{\ast A} + 2 \delta_{C \ast}$. Because $q$ and therefore $q_\infty^{}$
vanishes on nonpositive point measures, so that the outer sum ranges over all $x \in \widetilde \cX$ with $x|^{}_{\{1\}} = \texttt{C} \ast$ if $1 \in D(x)$ and $x|^{}_{\{2\}} = \ast \texttt{A}$ if $2 \in D(x)$. These are the types
$\texttt{CA}, \ast \texttt{A}, \texttt{C} \ast$ and $\epsilon = \ast \ast$. 
The inner sums are all $1$ for $x = \texttt{CA}$ and $x = \ast \ast$, and they are $-1$ for the remaining two types; this is because $S = \{1,2\}$ is the only subset of $S$ with at least two elements so that $\bar r_S^{} = 1$, and we have
\begin{equation*}
\begin{split}
& \binom{\nu^{\supseteq \{1,2\},D(\texttt{CA})}(\texttt{CA})}{2} \\
& \quad = \binom{\nu^{\supseteq \{1,2\},D(\ast \texttt{A})}(\ast \texttt{A})}{2} = 
\binom{\nu^{\supseteq \{1,2\},D(\texttt{C} \ast)}(\texttt{C} \ast)}{2} =
\binom{\nu^{\supseteq \{1,2\},D(\ast \ast)}(\ast \ast)}{2} \\
& \quad = \binom{2}{2} = 1; 
\end{split}
\end{equation*}
Thus, Eq.~\eqref{samplingformula} becomes
\begin{equation*}
q_1^{} (2 \delta_{\texttt{CA}}) = q_\infty^{} (2 \delta_{\texttt{C} \ast}  + 2 \delta_{\ast \texttt{A}}) - 
q_\infty^{} ( \delta_{\texttt{C} \ast} + 2 \delta_{\ast \texttt{A}} ) -
q_\infty^{} ( 2 \delta_{ \texttt{C} \ast} +  \delta_{\ast \texttt{A}} ) +
q_\infty^{} ( \delta_{\texttt{C} \ast} + \delta_{\ast \texttt{A}} ) .
\end{equation*}
\end{example}
We will give a graphical proof of this special case in Example~\ref{ex:continued} at the end of Subsection~\ref{subsec:couplinggraphical}, following a graphical version of the aforementioned coupling argument. This is independent of the measure-valued constructions to be introduced in the next section, and which are needed to prove Theorem~\ref{thm:maintheorem} in full generality. So, the reader who is looking to get a `feel' for these arguments before diving into the technical details is encouraged to skip ahead.

\begin{remark}
In this work, we consider a finite-alleles model of mutation. However, we stress that our arguments do not rely on the concrete form of the mutation model. In particular, the extension to the infinite alleles / infinite sites model would be straightforward.
\end{remark}

To prove Theorem~\ref{thm:maintheorem}, we start by introducing a slightly different version of the ARG which, albeit slightly more involved than the graphical construction at a first glance, is better suited to technical purposes. 

\section{A measure-valued coalescent with recombination}~\label{sec:marg}  
We have introduced the ARG as a random graphical construction and we have seen how it can be used to construct a random sample from a population that has evolved for a long time under recombination, mutation and genetic drift. In this section, our goal is to reformulate these ideas in a slightly different and more rigorous fashion.

In contrast to~\cite{JenkinsFearnheadSong2015}, who worked with untyped coalescents throughout, it will be more convenient for us to work with a \emph{typed} approach. This will make it much easier to later set up the coupling, as we will not have to explicitly keep track of the correspondence between the marginal coalescents in the ARG (see Section~\ref{sec:informalarg}) and the independent Kingman coalescents. Such bookkeeping was performed in~\cite{JenkinsFearnheadSong2015}, but seems difficult to generalise to more than two sites.

The main idea is to keep track in the ARG of the type configurations that lead to the desired sample. By exchangeability, it suffices to keep track of the frequencies in these configurations, which justifies our measure-valued approach below.

For instance, let us reconsider the example in Fig.~\ref{fig:arg}. The type frequencies at the leaves are given by
$
\nu = \delta_{(\texttt G, \texttt C, \texttt A)}^{} + \delta_{ (\ast, \texttt T, \texttt A) }^{} + \delta_{(\ast,\ast,\texttt A)}^{}
$,
where, as usual, $\ast$ is used to mark unobserved sites. We see that the first event that is encountered on the ancestral line of the leaf with type $(\ast, \texttt T, \texttt A)$ is the mutation
\begin{tikzpicture}[baseline=-1.4mm]
\draw (0,0) circle (6pt);
\node at (0,0) {$2$};
\node[anchor=west] at (0.1,0) {$\texttt C \to \texttt T$};
\end{tikzpicture}.
Therefore, in order to observe the desired allele \texttt T in this individual, there are \emph{two} possibilities for the allele just prior to (above) the mutation; it needs to be either \texttt T itself, or \texttt C because the \texttt C would have mutated to the desired \texttt T. 
If we stop the ARG right after that first mutation we would need to observe an ordered sample with frequencies
$\nu'_1 = \delta_{(\texttt G, \texttt C, \texttt A)}^{} + \delta_{(\ast,\texttt T, \texttt A)}^{} + \delta_{(\ast,\ast,\texttt A)}^{}$ or
$\nu'_2 = \delta_{(\texttt G, \texttt C, \texttt A)}^{} + \delta_{(\ast,\texttt C, \texttt A)}^{} + \delta_{(\ast,\ast,\texttt A)}^{}$.

In order to efficiently keep track of these multiple possibilities we need to allow for some ambiguity in the specification of alleles by working with \emph{fuzzy} types. For $i \in S$, we denote by $\widetilde X_i$ the set of nonempty subsets of $X_i$ and, in analogy to our earlier definitions, we define 
for $A \subseteq S$
\begin{equation*}
\widetilde X_A^{} \defeq \prod_{i \in A} \widetilde X_i^{}
\end{equation*}
and 
\begin{equation*}
\widetilde \cX \defeq \dot {\bigcup_{A \subseteq S} } \widetilde X_A^{}.
\end{equation*}
As before, $\widetilde X_\varnothing^{} = \{\epsilon \}$ where $\epsilon$ is the empty sequence. We call $\widetilde \cX$ the set of \emph{fuzzy types}. For a fuzzy $x \in \widetilde \cX$, the marginalisations $x|_B^{}$ and
the set $D(x)$ of observed sites and the alleles $x(i)$ at site $i$ are defined in analogy with the definitions in Section~\ref{sec:informalarg}. Moreover, we embed the set $\cX$ of (exact) types in $\widetilde \cX$ by identifying any $x \in \cX$ with $(\{x(i) \})_{i \in D(x)}^{} \in \widetilde \cX$. This also gives an embedding of $\cN(\cX)$ into $\cN(\widetilde \cX)$ and is our justification for using the same notations for fuzzy and exact types. 

The notions of compatibility and join need to be adapted slightly. We call $x, y \in \widetilde \cX$ \emph{compatible} if 
$x(i) \cap y(i) \neq \varnothing$ for all $i \in D(x) \cap D(y)$,  and \emph{incompatible}, otherwise. For compatible $x$ and $y$, their join $x \sqcup y$ is defined via
\begin{equation}\label{fuzzyjoin}
(x \sqcup y) (i) \defeq 
\begin{cases}
x(i) & \textnormal{for } i \in D(x) \setminus D(y), \\
y(i) & \textnormal{for } i \in D(y) \setminus D(x), \\
x(i) \cap y(i)   &\textnormal{for } i \in D(x) \cap D(y).
\end{cases}
\end{equation}
Of course, this is not well defined if $x$ and $y$ are incompatible
Note that this is consistent with our ealier definition for exact types. From now on, we will refer to fuzzy types simply as types unless mentioned otherwise.

Let us talk more systematically about the effects of the transitions in the (graphical) ARG on (the type frequencies of) configurations that lead to the desired sample. We start with mutation. Assume that we encounter a mutation
\begin{tikzpicture}[baseline=-1.4mm]
\draw (0,0) circle (6pt);
\node at (0,0) {$i$};
\node[anchor=west] at (0.1,-0.05) {$y_i^{} \to z_i^{}$};
\end{tikzpicture}
and that, just below the mutation, we want to see $x \in \widetilde \cX$. The reader should keep in mind that $x$ is a fuzzy type, whence $x(i)$ is a subset of $X_i$, while $y_i^{}, z_i^{}$ are single alleles, i.e. elements of $X_i$. In particular, statements of the form `$z_i^{} \in x(i)$'
are well-defined.

Let $m_i^{} (x;y_i^{},z_i^{})$ be the (fuzzy) type that needs to be observed just before the mutation to make this happen. Clearly, as mutation only affects site $i$, we must have $m_i^{} (x;y_i^{},z_i^{})(j) = x(j)$ for all 
$j \in D(x) \setminus \{i\}$. If $z_i^{} \in x(i)$, then observing $y_i^{}$ right before the mutation would work as well as observing any allele in $x_i^{}$. On the other hand, if $z_i^{} \not \in x(i)$, then $y_i^{}$ would not work, even if $y_i^{} \in x(i)$.
Thus,
\begin{equation}\label{mutantcandidates}
m_i^{} (x;y_i^{},z_i^{})(j) =
\begin{cases}
x(j) & \textnormal{if } j \neq i, \\
x(i) \cup \{y_i^{} \} & \textnormal{if } j = i \textnormal{ and }  z_i^{} \in x(i), \\
x(i) \setminus \{y_i^{} \} & \textnormal{if } j = i \textnormal{ and }   z_i^{} \not \in x(i).
\end{cases}
\end{equation}
Note that it might happen that $m_i^{} (x;y_i^{},z_i^{})(i) = \varnothing$. This would mean that it is impossible to observe our desired sample, and it makes no sense to trace its genealogy further. For this purpose, we introduce a cemetary state $\Delta$. We extend $\sigma$ from $\cN(\cX)$ to $\cN(\cX) \cup \{ \Delta \}$ via $\sigma(\Delta) = \Delta$.

The effect of coalescence is straightforward. If two lines with types $x$ and $y$ coalesce, we need to observe at the common ancestral line and at the sites in $D(x) \cap D(y)$ alleles that show up in both $x$ and $y$. If $x$ and $y$ are incompatible, this is not possible and we end up in $\Delta$. Otherwise, the type of the ancestral line will have to be $x \sqcup y$ as in Eq.~\eqref{fuzzyjoin}.

The effect of recombination is also easy to understand. In order to observe a certain  offspring type $x$, the fragments we need to observe at the parental lines are simply the marginals of $x$ with respect to the blocks of $\cA|_{D(x)}^{}$, with $\cA$ being the partition defining the fragmentation event. More concisely, 
the type configuration of the parental lines must have frequencies
\begin{equation*}
\sum_{\substack{A \in \cA \\ A \cap D(x) \neq \varnothing}} \delta_{x|_A^{}}^{}.  
\end{equation*}
To summarise:
\begin{definition}\label{def:measureARG}
The \emph{measure-valued} \textnormal{ARG (mARG)} 
%$\cR = (\cR^{}_t)_{t \geqslant 0}^{}$
is a continuous-time Markov chain on
\mbox{$\cN(\widetilde \cX) \cup \{ \Delta \}$} with the following transitions, starting from \mbox{$\nu \in \cN(\widetilde \cX)$}.
\begin{enumerate}
\item \textbf{Coalescence}: Independently for any $x, y \in \widetilde \cX$,
\begin{equation*}
\nu \to 
C_{x,y} (\nu) \defeq
\begin{cases}
\nu - \delta_{x}^{} - \delta_{y}^{} + \delta_{x \sqcup y}^{} & \textnormal{if } x \textnormal{ and } y \textnormal{ are compatible,}\\
\Delta & \textnormal{if } x \textnormal{ and } y \textnormal{ are incompatible.} 
\end{cases},
\end{equation*}
at rate $\nu(x) \nu(y)$ if $x \neq y$ and $\nu(x) (\nu(x) - 1)$ if $x = y$. 
\item \textbf{Recombination:} Independently for any $x \in \widetilde \cX$ and any partition $\cA$ of $S$,
\begin{equation*}
\nu \to \nu - \delta_x^{} + \sum_{\substack{A \in \cA \\ A \cap D(x) \neq \varnothing}} \delta_{x|_A^{}}^{}
\end{equation*}
at rate $\varrho_\cA^{} \nu(x)$.
\item \textbf{Mutation: } Independently for each $x \in \widetilde \cX$, each $i \in D(x)$ and each $y_i^{}, z_i^{} \in X_i$,  
\begin{equation*}
\nu \to 
\begin{cases}
\nu - \delta_x^{} + \delta_{m_i^{} (x;y_i^{},z_i^{})}^{} &  \textnormal{if } m_i^{} (x;y_i^{},z_i^{}) (i) \neq \varnothing, \\
\Delta & \textnormal{if } m_i^{} (x;y_i^{},z_i^{})(i) = \varnothing.
\end{cases}
\end{equation*}
at rate $\nu(x) u_i^{} M_i(y_i^{},z_i^{})$ and with $m_i^{} (x;y_i^{},z_i^{})$ as given in Eq.~\eqref{mutantcandidates}.
\end{enumerate}
\end{definition}
Generically, we will denote the mARG as $\cR = (\cR_t^{})_{t \geqslant 0}^{}$. 

\begin{remark}
The rate in 1 of Definition~\ref{def:measureARG} can in both cases be written more concisely as $\nu(x) ( \nu - \delta_x^{} )(y)$. This is the formulation we will use below.
\end{remark}

In order to connect this to the graphical construction in Section~\ref{sec:informalarg}, it is useful to think of the mARG (and of the split mARG, which will be introduced below) as a collection of evolving particles, corresponding to the unit Dirac measures; see Remark~\ref{rmk:asparticlesystems}. We say that a particle $\delta_x^{}$ is observed at $D(x)$.

Recall that in Section~\ref{sec:informalarg}, we constructed a sample at stationarity by running the ARG until the time of the most recent common ancestors and then assigning types to them, sampled independently according to the stationary distributions of the mutation chains.
Here, this means that we let the mARG run until it is either in the cemetary state $\Delta$ (meaning that we can be sure that the genealogy is incompatible with the desired sample) or \emph{simple} in the following sense.
\begin{definition} \label{def:simpleness}
We call $\nu \in \cN (\widetilde \cX)$ \emph{simple} if 
\begin{enumerate}
\item
$A = B$ or $A \cap B = \varnothing$ for all $A,B \subseteq S$ with $\nu(\widetilde X_A) >0$ and $\nu(\widetilde X_B) > 0$, and
\item
$\nu(\widetilde X_A) \leqslant 1$ for all $A \subseteq S$. 
\end{enumerate}
\end{definition}
Item 1 means that the sets of observed sites of different indviduals are mutually disjoint, while 2 guarantees that each set is observed at most once. Taking this together, this corresponds to the situation at time $T_{\textnormal{MRCA}}$ in Fig.~\ref{fig:arg}, where the sets of sites for which there are ancestral lines are disjoint, and for each set $A$ of sites, there is at most $1$ line ancestral to $A$. 

In particular, if $\nu = \Delta$ we already know that $q(\nu) = 0$.
For simple $\nu$, $T_{\textnormal{MRCA}} = 0$ and $q(\nu)$ is just the probability of drawing the desired types from the product measures of $\pi_i^{}$ associated with the ancestors, which are in this case just the sampled individuals themselves.

If $\nu$ is not simple and
$\nu \neq \Delta$, we run the mARG $\cR^\nu$ (the superscript indicating the initial condition) until the stopping time 
\begin{equation} \label{stoppingtimeT}
T \defeq \min \{ t \geqslant 0 : \cR_t^{} \textnormal{ is simple or } \cR_t^{} = \Delta \}.
\end{equation}
Then, $\cR^\nu_T$ is either simple or $\Delta$, and the sampling probability of $\nu$ \emph{conditional on this realisation of the genealogy} is given by $q(\cR_T^\nu)$. To summarise,

\begin{definition}\label{def:formalq}
For $\nu$ simple or $\nu = \Delta$, we let
\begin{equation} \label{simplesample}
q(\nu) \defeq 
\begin{cases}
\prod_{x \in \supp (\nu)} \prod_{i \in D(x)} \pi_i^{} (x|_i^{}(i)) & \textnormal{if } \nu \textnormal{ is simple}, \\
0 & \textnormal{if } \nu = \Delta.
\end{cases}
\end{equation}
Otherwise, we set 
\begin{equation}  \label{nonsimplesample}
q(\nu) \defeq \EE \big [ q(\cR_T^{}) | \cR_0^{} = \nu  \big ].
\end{equation}
For consistency, we also let $q(0) \defeq 1$.
\end{definition}
Let us stress once more that this definition reflects precisely the construction in Section~\ref{sec:informalarg}, where we construct a realisation of the ARG until time $T_{\textnormal{MRCA}}$ which roughly corresponds to our $T$; the difference is that here in the measure-valued setting, we may stop early once we discover that the sample is unobtainable, i.e. when the mARG hits $\Delta$. The same comment applies to Def.~\ref{def:formalqinfty} below.

\begin{remark}\label{rmk:duality}
The definition of $q$ in Eq.~\eqref{nonsimplesample} very loosely resembles a duality relation for Markov processes; see~\cite{kurtjansen} for a survey. It seems like an interesting  problem to try and find a dual diffusion process, describing the evolution forward in time. For two sites, such a diffusion process was found in~\cite{JenkinsFearnheadSong2015}, although no formal duality relation was established. It was based on expressing type frequencies in terms of marginal frequencies and linkage disequilibrium (i.e., correlation between the two sites), which linearised the problem. For the case of more than two sites considered here, a possible starting point might be the linearisation discussed in~\cite{haldane} which works for an arbitrary number of sites and general patterns of recombination.
\end{remark}

As we hinted at in Section~\ref{sec:informalarg}, the central device for our proof of Theorem~\ref{thm:maintheorem} will be a coupling between the ARG and a collection of independent Kingman coalescents to approximate the ancestries of the individual sites. So far, we have introduced the mARG as a measure-valued representation of the ARG and used it to give a precise definition of the sampling probability $q$. Next, we define a similar representation for the collection of independent Kingman coalescents. 

The most obvious way to do this would be to use a collection $(\cR^{(i)})_{i \in S}$ of independent mARGs with respective initial configurations
$(\nu^{ \{ i \}})_{i \in S}$. However, it is more convenient to summarise them into a single process, which is essentially a ``version'' of the mARG that is supported on 
\begin{equation} \label{x1definition}
\widetilde \cX_1 \defeq  \{ x \in \widetilde \cX : |D(x)| = 1  \}
\end{equation}
and in which $x$ and $y$ with $D(x) \neq D(y)$ cannot merge, i.e. transitions of the form $\nu \to C_{x,y} (\nu)$ are prohibited.
Incidentally, this exactly captures the behaviour of the (m)ARG as $\varrho \to \infty$. While lines ancestral to different sites are allowed to merge, they are split immediately by recombination and so these mergers become ``invisible'' in the large-recombination limit.

%In line with our goal of deriving an approximation for $q$ for large $\varrho$, we now define a version of $\cR$ ``with infinite recombination rate'', denoted by $%\cR^{\infty} = (\cR_t^{\infty})_{t \geqslant 0}^{}$. As we discussed already in Section~\ref{sec:informalarg} in the untyped setting, the effect
%of `infinite' recombination is that any particle $\delta_x^{}$ with $D(x) \geqslant 2$ is split immediately into
%the fragments $\sum_{i \in D(x)} \delta_{x|^{}_{i}}^{}$. This also means that we will in neglect in $\cR^\infty$ coalescence between particles observed at %different sites because they would be split again immediately. \todo{this paragraph needs to be rewritten \ldots the reviewer did not like fragments \smiley.}

\begin{definition}\label{def:sm-arg}
The \emph{split} \textnormal{mARG} (\textnormal{s-mARG}) is a Markov chain in continuous time on \mbox{$\cN(\widetilde \cX_1^{}) \cup \{ \Delta \}$} with
$\widetilde \cX_1^{}$ given in Eq.~\eqref{x1definition},
performing the following transitions starting from $\nu$.
\begin{enumerate}
\item \textbf{Coalescence:} Independently for any $i \in S$ and $x^i, y^i \in \widetilde X^{}_{\{i\}}$,
\begin{equation*}
\nu \to C_{x^i, y^i} (\nu)
\end{equation*}
at rate $\nu(x^i) (\nu - \delta_{x^i})(y^i)$, with $C_{x^i,y^i}$ as in Definition~\ref{def:measureARG}.
\item \textbf{Mutation:} Independently for each $i \in S$ and $x^i \in \widetilde X_{\{i\}}$, and $y_i^{}, z_i^{} \in X_i$, 
\begin{equation*}
\nu \to 
\begin{cases}
\nu - \delta_{x^i}^{} + \delta_{m_i^{}(x^i;y_i^{},z_i^{})}^{} & \textnormal{if } m_i^{} (x^i;y_i^{},z_i^{})(i) \neq \varnothing, \\
\Delta & \textnormal{if } m_i^{} (x^i;y_i^{},z_i^{})(i) = \varnothing,
\end{cases}
\end{equation*}
at rate $u_i^{} M(y_i^{},z_i^{}) \nu(x^i)$ with $m_i^{}(x^i;y_i^{},z_i^{})$ as in Eq.~\eqref{mutantcandidates}.
\end{enumerate}
\end{definition}
As mentioned above, the s-mARG is the natural limit of the mARG as $\varrho \to \infty$ so that we will generically write 
$\cR^\infty = (\cR^\infty_t)_{t \geqslant 0}$
for a realisation of the s-mARG. Moreover, it is clear that an s-mARG $\cR^\infty$ started from $\nu \in \cN(\widetilde \cX)$ can be written as $\cR^1_t + \ldots + \cR^n_t$, where 
$\cR^1, \ldots, \cR^n$ are independent mARGs, started from $\nu^{\{1\}}, \ldots, \nu^{\{n\}}$.

We now use $\cR^\infty$ to give a precise definition of $q_\infty^{}$. This is entirely analogous to Definition~\ref{def:formalq}, except that we use $\cR^\infty$ rather than $\cR$ in order to play nonsimple sample configurations back to either simple ones or $\Delta$. In particular, 
$q_\infty^{} (\Delta) = q (\Delta) = 0$ and $q_\infty^{} (\nu) = q (\nu)$ for simple $\nu$. In analogy with Eq.~\eqref{stoppingtimeT}, let
\begin{equation*}
T^\infty \defeq \min \{ t \geqslant 0 : \cR_t^{\infty} \textnormal{ is simple or } \cR_t^{\infty} = \Delta \}.
\end{equation*}

\begin{definition}\label{def:formalqinfty}
For $\nu$ simple or $\nu = \Delta$, we let
\begin{equation} \label{simplesampleq}
q_\infty^{} (\nu) \defeq 
\begin{cases}
\prod_{x \in \supp (\nu)} \prod_{i \in D(x)} \pi_i^{} (x|_{ i }^{}(i)) = q(\nu) & \textnormal{if } \nu \textnormal{ is simple}, \\
0 & \textnormal{if } \nu = \Delta.
\end{cases}
\end{equation}
Otherwise, we set
\begin{equation} \label{nonsimplesampleq}
q_\infty^{} (\nu) \defeq \EE \big [ q_\infty^{}(\cR_{T^\infty}^\infty) | \cR_0^\infty = \sigma(\nu)  \big ].
\end{equation}
\end{definition}
Note that in contrast to Definition~\ref{def:formalq}, where we started $\cR$ from $\nu$, we let $\cR^\infty$ start from 
$\sigma (\nu) =  \nu^{\{1\}} + \ldots + \nu^{\{n\}}$.
From this definition, it is not hard to show that $q_\infty^{}$ factorises as mentioned earlier (see Eq.~\eqref{qinfty})
\begin{lemma} \label{lem:qvsqinfty}
For all $\nu \in \cN(\widetilde \cX)$,
\begin{equation*}
q_\infty^{} (\nu) = \prod_{i \in S} q(\nu^{\{i\}}).
\end{equation*}

\end{lemma}
\begin{proof}
Recalling the discussion below Definition~\ref{def:measureARG}, for each $t \geqslant 0$, $\cR^\infty$ started from $\sigma(\nu)$ with $\nu \in \widetilde \cX_1$ can be written as $\cR^1_t + \ldots + \cR^n_t$ where $\cR^i$ is an mARG started from $\nu^{\{i\}}$ and all the 
$\cR^i$ are independent. Assuming for now that $ \cR_{T^\infty}^\infty \neq \Delta$, we have in particular that
\begin{equation*}
\supp \big ( \cR_{T^\infty}^\infty \big ) = \supp \big (  \cR_{T^\infty}^1   \big )  \dot \cup  
\ldots \dot \cup  \supp \big ( \cR_{T^\infty}^n \big ).
\end{equation*}
Moreover, by the definition of simplicity, $\supp \big ( \cR^i_{T^\infty}  \big )$ is either empty or a singleton. If it is not empty, we denote its unique element by $x_i^{} \in \widetilde X_i$. With this, applying Eqs.~\eqref{simplesampleq} and~\eqref{simplesample} yields
\begin{equation*}
q_\infty^{} ( \cR_{T^\infty}^\infty) = \prod_{i} \pi_i^{} (x_i^{}) = \prod_i q^{} (\cR_{T^\infty}^i)
\end{equation*}
where the product runs over all $i$ for which $\supp \big ( \cR^i_{T^\infty}  \big )$ is not empty. Note that this also holds if 
$\cR_{T^\infty}^\infty = \Delta$, because then we also have $\cR_{T^\infty}^i = \Delta$ for some $i$ and both sides are $0$.
Therefore, by taking the expectation,
\begin{equation*}
q_\infty^{} (\nu) = \EE \big [  q_\infty^{} ( \cR_{T^\infty}^\infty)         \big ] = \EE \big [   \prod_i q^{} (\cR_{T^\infty}^i)         \big ]
=  \prod_i \EE \big [   q^{} (\cR_{T^\infty}^i)         \big ]
     \prod_i \EE \big [   q^{} (\cR_{T^i}^i)         \big ].
\end{equation*}
Here, $T^i$ is the instance of the stopping time $T$ defined in Eq.~\eqref{stoppingtimeT} associated with $\cR^i$. In the last step, we are able to replace $T^\infty$ by $T^i$ because $T^i \leqslant T^\infty$ and between $T^i$ and $T^\infty$, $\cR^i$ is simple and only affected by mutation; recall that $\pi_i^{}$ in Eq.~\eqref{simplesample} is precisely the stationary distribution of the mutation chain at site $i$.

\end{proof}

\section{The central coupling}\label{sec:coupling}
As already hinted at in Section~\ref{sec:informalarg}, we want to approximate the ARG by a collection of independent Kingman coalescents, one for each site.
We will call them \emph{almost ancestral graphs} (aAGs) and write $\textnormal{aAG}_i$ for the \emph{\textnormal{aAG}  associated with site $i$.}
We will couple the ARG and the aAGs in such a way that the aAGs agree with the marginal genealogies ($(\textnormal{MC}_i)_{i \in S})$) in the ARG with high probability. Eventually, investigating the unlikely event that they don't will give us the first-order correction in Eq.~\eqref{samplingformula}.

We first give a graphical, slightly informal version of our construction to provide some intuition, similar in spirit to the informal description of the ARG in 
Section~\ref{sec:informalarg}. In Subsection~\ref{subsec:couplingrigorous}, we will pave the way for the proof of Theorem~\ref{thm:maintheorem} by translating our ideas into a coupling between the mARG and the s-mARG from Definitions~\ref{def:measureARG} and \ref{def:sm-arg}.  

\subsection{Graphical heuristics} \label{subsec:couplinggraphical}
The following construction is illustrated in Fig.~\ref{fig:graphicalcoupling}.
Initially, 
for each site $i \in S$, we start the $\textnormal{aAG}_i$ with one line for each line in the starting configuration in the ARG that is ancestral to $i$. 
Consequently, we initially have for each $i \in S$ a natural bijective correspondence
\begin{equation} \label{linecorrespondence} 
\Phi_i  : \{ \textnormal{Lineages in the ARG ancestral to }  i \} \to \{ \textnormal{Lineages in the } \textnormal{aAG}_i    \};
\end{equation}
we also say that the ARG and the aAGs are \emph{matched}.

We want to define the joint dynamics of the ARG and the aAGs so that this correspondence is preserved, that the ARG and aAGs \emph{remain} matched for as long as possible. To this end, it is useful to think of transitions in the aAGs occurring in response to transitions in the ARG. Recombination in the ARG does not produce a transition in the aAGs because recombination does not affect the law of marginal lineages (although it affects their dependence structure!). Also, mutation events are easy to match because they occur independently on each site.

The matching of coalescence events is slightly more subtle. 
We write $k,\ell$ for generic lineages in the ARG and $A(k), A(\ell)$ for the set of sites they are ancestral to. 
If $A(k) \cap A(\ell) = \{i \}$ for some $i \in S$, we can match a coalescence of the pair $(k,\ell)$ by letting the associated pair $(\Phi_i(k), \Phi_i(\ell))$ in the 
$\textnormal{aAG}_i$ coalesce at the same time (while updating the correspondence~\eqref{linecorrespondence} accordingly).
If $A(k) \cap A(\ell) = \varnothing$, the aAGs don't perform any transition.

On the other hand, if $A(k) \cap A(\ell) \supseteq \{i,j\}$ for some $i,j \in S$, $i \neq j$, we would need to let the pairs $(\Phi_i (k), \Phi_i (\ell))$ and 
$(\Phi_j (k), \Phi_j (\ell))$ coalesce at the same time in order to preserve the correspondence of lineages between the ARG and the aAGs. But then, the evolution of the aAGs would no longer be independent. Thus, we call such a pair $(k,\ell)$ a \emph{bad pair (in the ARG)}
and its coalescence a \emph{bad coalescence (in the ARG)}.
We refrain from trying to match bad coalescences in the ARG to transitions in the aAGs. Instead, at each time at which a bad coalescence occurs in the ARG, no transition is performed by the aAGs. Acknowledging that the aAGs then no longer describe the sample genealogy, we simply let them from that time onward evolve as 
independent Kingman coalescents and  independently of the ARG. We say that the ARG and the aAGs are \emph{unmatched}.

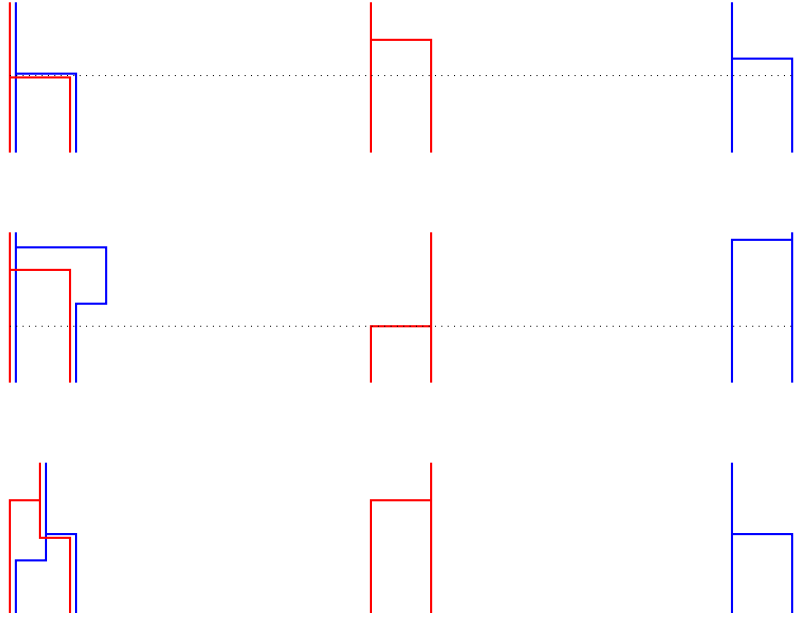
\begin{figure}[t]
\centering
\begin{tikzpicture}[xscale=0.8, yscale=0.5]
%ARG
\draw[red, thick] (0,0) -- (0,4);
\draw[blue, thick] (0.1,0) -- (0.1,4);
\draw[red, thick] (1,0) -- (1,2) -- (0,2);
\draw[blue, thick] (1.1,0) -- (1.1,2.1) -- (0.1,2.1);
%aAG1
\draw[red, thick] (6,0) -- (6,4);
\draw[red, thick] (7,0) -- (7,3) -- (6,3);
%aAG2
\draw[blue, thick] (12,0) -- (12,4);
\draw[blue, thick] (13,0) -- (13,2.5) -- (12, 2.5);
%unmatching
\draw[dotted] (0,2.05) -- (13,2.05);
\end{tikzpicture}
\\[1cm]
\begin{tikzpicture}[xscale=0.8, yscale=0.5]
%ARG
\draw[red, thick] (0,0) -- (0,4);
\draw[blue, thick] (0.1,0) -- (0.1,4);
\draw[red, thick] (1,0) -- (1,3) -- (0,3);
\draw[blue, thick] (1.1,0) -- (1.1,2.1) -- (1.6,2.1) -- (1.6,3.6) -- (0.1,3.6);
%aAG1
\draw[red, thick] (6,0) -- (6,1.5) -- (7,1.5);
\draw[red, thick] (7,0) -- (7,4);
%aAG2
\draw[blue, thick] (12,0) -- (12,3.8) -- (13,3.8);
\draw[blue, thick] (13,0) -- (13,4);
%unmatching
\draw[dotted] (0,1.5) -- (13,1.5);
\end{tikzpicture}
\\[1cm]
\begin{tikzpicture}[xscale=0.8, yscale=0.5]
%ARG
\draw[red, thick] (0,0) -- (0,3) -- (0.5,3);
\draw[blue, thick] (0.1,0) -- (0.1,1.4) -- (0.6,1.4) -- (0.6,4);
\draw[red, thick] (1,0) -- (1,2) -- (0.5,2) -- (0.5,4);
\draw[blue, thick] (1.1,0) -- (1.1,2.1) -- (0.6,2.1);
%aAG1
\draw[red, thick] (6,0) -- (6,3) -- (7,3);
\draw[red, thick] (7,0) -- (7,4);
%aAG2
\draw[blue, thick] (12,0) -- (12,4);
\draw[blue, thick] (13,0) -- (13,2.1) -- (12,2.1);
%no unmatching
\end{tikzpicture}
\caption{\label{fig:graphicalcoupling}
An illustration of the coupled construction of the ARG (left) and the aAGs for two sites. The $\textnormal{aAG}_1^{}$ is drawn in red, the $\textnormal{aAG}_2^{}$ in blue.  Top: a bad coalescence occurs in the ARG. Middle: a bad coalescence occurs in the $
\textnormal{aAG}_1^{}$. In both cases, the dotted line indicates the time when the ARG and the aAGs unmatch. Bottom: no bad coalescence occurs, and the lineages remain matched for all time.
}
\end{figure}

Moreover, while the ARG and the aAGs are still matched, for any bad pair $(k,\ell)$ in the ARG and any $i \in  A(k) \cap A(\ell)$, we also call the corresponding pair $(\Phi_i (k), \Phi_i (\ell))$ in the $\textnormal{aAG}_i$ a bad pair (in the $\textnormal{aAG}_i^{}$).
Note that these are precisely the pairs of lines in the $\textnormal{aAG}_i^{}$ that do not coalesce in response to coalescence in the ARG. 
 In order to obtain a collection of independent Kingman coalescents, we postulate that each bad pair in the $\textnormal{aAG}_i^{}$ coalesces independently at rate $1$ and independently of the ARG in what we call a \emph{bad coalescence event (in the  $\textnormal{aAG}_i^{}$)}. Also in this case, we unmatch the ARG and the aAGs, and let them continue to evolve independently without the ARG performing a transition of his own. All other pairs of lineages (either in the ARG or in the aAGs) are called \emph{good} pairs and we refer to coalescence between them as \emph{good coalescence}. 

Let us stress once more that good coalescence events, although they are matched between the ARG and the aAGs, occur in at most one aAG, namely the one corresponding to the single site that the coalescing lineages in the ARG have in common. Hence, $\textnormal{aAG}_1, \ldots, \textnormal{aAG}_n$ evolve independently of each other, but not independently of the ARG.
We close this section by summarising the joint evolution of the ARG and the aAGs.
\begin{enumerate}
\item
The family $(\textnormal{aAG}_i)_{i \in S}$ is independent. 
\item
For each $i \in S$, a mutation mark
\begin{tikzpicture}[baseline=-1.4mm]
\draw (0,0) circle (6pt);
\node at (0,0) {$i$};
\node[anchor=west] at (0.1,-0.1) {$x_i^{} \to y_i^{}$};
\end{tikzpicture} 
appears on each lineage in the $\textnormal{aAG}_i$ independently and independently for each choice of $x_i^{}, y_i^{} \in X_i$ at rate $u_i^{} M_i(x_i^{},y_i^{})$.  
\item
Each (ordered) pair of lineages in the $\textnormal{aAG}_i$ coalesces at rate $1$.
\item
Prior to the occurrence of a bad coalescence (of either type), the ARG and the aAGs are matched, that is, the aAGs agree with the marginal coalescents embedded in the ARG. With the occurrence of a bad coalescence the ARG and the aAGs are unmatched, and the aAGs keep evolving as in 1 to 3, but independently of the ARG.
\end{enumerate}

We close by finishing Example~\ref{ex:concreteexample}.
\begin{example}[Example~\ref{ex:concreteexample} continued] \label{ex:continued}
Recall that we want to verify the approximation 
\begin{equation*}
q(2 \delta_{\texttt C  \texttt A}) = q_\infty^{} (2 \delta_{\texttt C \texttt A}) + \varrho^{-1} q_1^{} (2 \delta_{\texttt C \texttt A}) + \cO(\varrho^{-2}) 
\end{equation*}
with
\begin{equation*}
q_1^{} (2 \delta_{\texttt C \texttt A})  = q_\infty^{} (2 \delta_{\texttt{C} \ast}  + 2 \delta_{\ast \texttt{A}}) - 
q_\infty^{} ( \delta_{\texttt{C} \ast} + 2 \delta_{\ast \texttt{A}} ) -
q_\infty^{} ( 2 \delta_{ \texttt{C} \ast} +  \delta_{\ast \texttt{A}} ) +
q_\infty^{} ( \delta_{\texttt{C} \ast} + \delta_{\ast \texttt{A}} ).
\end{equation*}
We let the $\textnormal{aAG}_1$ and $\textnormal{aAG}_2$ be coupled as just explained and write $\cS_1$,$\cS_2$ for the ordered samples generated by them. Similarly, we write S for the random sample generated by the ARG, more precisely, the (random) list of samples generated by the ARG. With this,
\begin{equation*}
q (2 \delta_{\texttt C \texttt A}) = \PP [ \cS = (\texttt C \texttt A, \texttt C \texttt A) ]
\end{equation*}
and
\begin{equation*}
q_\infty^{} (2 \delta_{\texttt C \texttt A}) = \PP [ \cS_1 = (\texttt C \ast, \texttt C \ast) ] \PP [ \cS_2 = (\ast \texttt A, \ast \texttt A) ].
\end{equation*}
As $\varrho$ becomes large, the most likely event is that no bad coalescence occurs, which we call $E$. On $E$, we have that 
$\cS = \cS_1 \sqcup \cS_2$, where the concatenation is defined componentwise (e.g. 
\mbox{$(\texttt C \ast, \texttt C \ast) \sqcup (\ast \texttt A, \ast \texttt A) = (\texttt C \texttt A, \texttt C \texttt A)$}).

The next likely event is that a bad coalescence occurs prior to any transition other than recombination in the ARG (and therefore also prior to any transition in the aAGs). We write $F_0,F_1,F_2$ for the events that it occurs in the ARG, in the $\textnormal{aAG}_1$ or in the $\textnormal{aAG}_2$. 
The events $F_0, F_1, F_2$ all occur with probability $1/\varrho + \cO(\varrho^{-2})$ because 
we see a coalescence at rate $2$ (recall that each \emph{ordered} pair coalesces at rate $1$), mutation at some constant rate $C$ and recombination at rate
$2 \varrho$ ($\varrho$ per line). All other events have probability of order $\varrho^{-2}$. Roughly speaking, this is because all other events require transitions that are not recombinations to occur while there is at least one lineage ancestral to both sites. And the total time before $T_{\textnormal{MRCA}}$ for which there is such a lineage is of order $\varrho^{-1}$. 

We let $F \defeq F_0 \cup F_1 \cup F_2$. Then by decomposing the representation above, we have
\begin{equation*}
q(2 \delta_{\texttt C \texttt A}) = \PP [ \cS = (\texttt C \texttt A, \texttt C \texttt A), E ] + \PP [ \cS = (\texttt C \texttt A, \texttt C \texttt A), F  ] + \cO(\varrho^{-2})
\end{equation*}
and, analogously,
\begin{equation*}
q_\infty^{} (2 \delta_{\texttt C \texttt A}) = \PP [ \cS_1 \sqcup \cS_2 = (\texttt C \texttt A, \texttt C \texttt A), E ] + \PP [ \cS_1\sqcup \cS_2 = (\texttt C \texttt A, \texttt C \texttt A), F  ] + \cO(\varrho^{-2}).
\end{equation*}
Solving for $\PP [ \cS_1 \sqcup \cS_2 = (\texttt C \texttt A, \texttt C \texttt A), E ]$ and using that $\cS = \cS_1 \sqcup \cS_2$ on $E$ then gives
\begin{equation} \label{exampledecomp}
q (2 \delta_{\texttt C \texttt A}) = q_\infty^{} (2 \delta_{\texttt C \texttt A}) + \PP [ \cS = (\texttt C \texttt A, \texttt C \texttt A), F      ]
- \PP [ \cS_1 \sqcup \cS_2 = (\texttt C \texttt A), F         ] + \cO(\varrho^{-2}).
\end{equation}
Decomposing $F = F_0 \dot \cup F_1 \dot \cup F_2$ gives
\begin{equation*}
\begin{split}
&\PP [ \cS = (\texttt{CA}, \texttt{CA}), F] \\
&\quad =  \PP [ \cS = (\texttt{CA}, \texttt{CA}), F_0] + \PP [ \cS = (\texttt CA, \texttt CA), F_1] + \PP [ \cS = (\texttt CA, \texttt CA), F_2] \\ 
&\quad = \varrho^{-1} ( \PP [\cS = \texttt C \texttt A] + 2 \PP [\cS = (\texttt{CA}, \texttt{CA})]) \\
&\quad = \varrho^{-1} (q(\delta_{\texttt{CA}}) + 2 q(2 \delta_{\texttt{CA}}) ) \\
&\quad = \varrho^{-1} (q_\infty^{}(\delta_{\texttt{CA}}) + 2 q_\infty^{}(2 \delta_{\texttt{CA}}) ) + \cO(\varrho^{-2}) \\
&\quad = \varrho^{-1} (2 q_\infty^{}  (2 \delta_{\texttt{C} \ast}  + 2 \delta_{\ast \texttt{A}}) + q_\infty^{} (\delta_{\texttt C \ast} + \delta_{\ast \texttt A})).
\end{split}
\end{equation*}
In the last step, we used that $q_\infty^{}(\delta_{\texttt{CA}})$ is the sampling probability associated with the independent aAGs, and therefore only depends on the frequencies of alleles at individual sites; see Lemma~\ref{lem:qvsqinfty}. To see the second equality, we can argue by restarting the ARG. Note that on $F_0$, we see a coalescence between the two lines in the ARG, which are ancestral to both sites. By the propagation rules explained earlier, we see the type $\texttt{CA}$ at each of them, if and only if the parental line has this type. So, we restart the ARG with a single line (ancestral to both sites) and the probability that we observe the type 
$\texttt{CA}$ at this line is
$q(\delta_{\texttt{CA}})$. On the other hand, if the bad coalescence occurs in either of the aAGs, the ARG is unaffected, and we start again with two lines at which we both observe the type $\texttt{CA}$ with probability $q(2 \delta_{\texttt{CA}})$.
Similarly, 
\begin{equation*}
\begin{split}
&\PP [\cS_1 \sqcup \cS_2 = (\texttt{CA}, \texttt{CA}), F    ] \\
& \quad = \PP   [\cS_1 \sqcup \cS_2 = (\texttt{CA}, \texttt{CA}), F_0   ] +  \PP [\cS_1 \sqcup \cS_2 = (\texttt{CA}, \texttt{CA}), F_1  ] 
+  \PP [\cS_1 \sqcup \cS_2 = (\texttt{CA}, \texttt{CA}), F_2  ] \\
&\quad = \varrho^{-1} ( q_\infty^{} ( 2 \delta_{\texttt{C} \ast} + 2 \delta_{\ast  \texttt{A}}) + q_\infty^{} (  \delta_{\texttt{C} \ast} + 2 \delta_{\ast  \texttt{A}}) +  q_\infty^{} ( 2 \delta_{\texttt{C} \ast} +  \delta_{\ast  \texttt{A}})   ),      
\end{split}
\end{equation*}
yielding the claim after insertion into Eq.~\eqref{exampledecomp}.
\end{example}

\subsection{Measure-valued formulation}
\label{subsec:couplingrigorous}

%%%MOAR NEWER%%%
Recall that in Section~\ref{sec:marg}, we introduced the measure-valued ARG (mARG) as a way of tracking the backward-in-time evolution of the ARG, or, more precisely, the types that have to be carried by the lineages in the ARG in order to yield a given sample configuration. Similarly, the aAGs are represented by the s-mARG. 
We will therefore construct a coupling $(\widetilde \cR, \widetilde \cR^\infty)$ between an mARG $\widetilde \cR$ and an s-mARG 
$\widetilde \cR^\infty$ such that $\sigma(\widetilde \cR_t) = \widetilde \cR^\infty_t$ as long as no bad coalescence occurred before time $t$, that is, as long as the underlying ARG and the aAGs are still matched.

In fact, we will construct a coupling $(\widetilde \cR, \widetilde \cR^\infty, \fm)$ where $\fm = (\fm_t^{})_{t \geqslant 0}^{}$ takes values in $0$ and $1$ and keeps track of whether the ARG and the aAGs are still matched at any given point in time; if 
$\fm_t = 1$, this means that no bad coalescence occurred until time $t$, $\widetilde \cR$ and $\widetilde \cR^\infty$ are still matched and, in particular,
$\sigma(\widetilde \cR_t^{}) = \widetilde \cR^\infty_t$. In contrast, if $\fm_t = 0$, then the ARG and the aAGs have been unmatched before time $t$.

First, we describe the transitions when $\widetilde \cR$ and $\widetilde \cR^\infty$ are matched ($\fm_t^{} = 1$), and 
$(\widetilde \cR, \widetilde \cR^\infty)$ is in a state of the form $(\nu,\sigma(\nu))$.
When they are unmatched ($\fm_t^{} = 0$), they will simply evolve independently of each other according to Defs.~\ref{def:measureARG} and~\ref{def:sm-arg}. 
Recall from Definition~\ref{def:measureARG} that, starting from state $\nu$, $\widetilde \cR$ performs independently for each $x,y \in \widetilde \cX$ a transition of the form $\nu \to C_{x,y}^{} (\nu)$ at rate
$\nu(x) (\nu - \delta_x)(y)$; any ordered pair of lines in the ARG coalesces at rate $1$, and $\nu(x) (\nu - \delta_x)(y)$ is the total number of ordered pairs of lines that have carry the (fuzzy) types $x$ and $y$ in order to produce the desired sample. 

As in the previous subsection, if $D(x) \cap D(y) = \{i\}$ for some $i \in S$, such a coalescence is called \emph{good} and matched by a transition in 
$\widetilde \cR^\infty$
of the form $\sigma(\nu) \to C_{x|_i^{}, y|_i^{}}^{} (\sigma(\nu))$. Keep in mind that whenever $D(x) \cap D(y) = \{i\}$, $x|_i^{}$ and $y|_i^{}$ are compatible if and only if $x$ and $y$ are. 
If $D(x) \cap D(y) = \varnothing$, nothing happens in $\widetilde \cR^\infty$.  In this case, $\fm$ stays at $1$.

On the other hand, if $|D(x) \cap D(y)| \geqslant 2$, $\widetilde \cR$ still performs transitions (corresponding to bad coalescence in the ARG) of the form 
$\nu \to C_{x,y} (\nu)$, while $\widetilde \cR^\infty$ stays put. In this case, $\fm$ jumps to $0$ to indicate that 
$\widetilde \cR$ and $\widetilde \cR^\infty$ are now unmatched. We call such transitions \emph{bad coalescences in $\widetilde \cR$}

So far, the rate of coalescence in $\widetilde \cR^\infty$ is too low compared to the rate in Def.~\ref{def:sm-arg}; as in the untyped, graphical setting of Subsection~\ref{subsec:couplinggraphical}, we therefore have to add
\emph{bad coalescence events in $\widetilde \cR^\infty$} to make up for this. More concretely, fix 
$i \in S$ and $x^i, y^i \in \widetilde X_{\{ i \} }$. 
According to Definition~\ref{def:sm-arg}, we want transitions $\sigma(\nu) \to C_{x^i, y^i} (\sigma(\nu))$ to occur at rate
\begin{equation*}
\sigma(\nu)(x^i) \big ( \sigma(\nu)   - \delta_{x^i}  \big ) (y^i) = 
\sum_{\substack{x \in \widetilde \cX \\ i \in D(x)}} \one_{x|_i^{} = x^i} \nu(x) 
\Big ( \sum_{\substack{y \in \widetilde \cX \\ i \in D(y)}} \one_{y|_i^{} = y^i } \nu(y) - \delta_{x^i}(y^i)  \Big).
\end{equation*}
Observing that
\begin{equation*}
\begin{split}
\sum_{\substack{x,y \in \widetilde \cX \\ i \in D(x) \cap D(y)}} &\one_{x|_i^{} = x^i} \one_{y|_i^{} = y^i} \nu(x) \delta_x(y) \\
&=
\sum_{\substack{x \in \widetilde \cX \\ i \in D(x)} } \one_{x|_i^{} = x^i} \one_{x|_i^{} =  y^i} \nu(x) \\
&=
\sum_{\substack{x \in \widetilde \cX \\ i \in D(x)} } \delta_{x^i} (y^i) \one_{x|_i^{} = x^i} \nu(x) \\
&=
\delta_{x^i}(y^i) \sigma(\nu) (x^i),
\end{split}
\end{equation*}
this is equal to
\begin{equation} \label{soll}
\sum_{\substack{x,y \in \widetilde \cX \\ i \in D(x) \cap D(y)}} \one_{x|_i^{} = x^i} \one_{y|_i^{} = y^i}  \nu(x) (\nu - \delta_x^{})(y)
\end{equation}
However, the rate at which such transitions happen in response to good coalescence in $\widetilde \cR$ is merely
\begin{equation} \label{ist}
\sum_{\substack{x,y \in \widetilde \cX \\ D(x) \cap D(y) = \{i\} }} \one_{x|_i^{} = x^i}^{}  \one_{y|_i^{} = y^i}^{} \nu(x) (\nu - \delta_x) (y).
\end{equation}
Subtracting \eqref{ist} from \eqref{soll}, we see that we need bad coalescences (in $\widetilde \cR^\infty$) between $x^i$ and $y^i$, that is, 
transitions of the form 
\begin{equation*}
(\nu,\sigma(\nu),0) \to (\nu, C_{x^i,y^i} (\sigma(\nu)) , 1)
\end{equation*}
to occur at rate 
\begin{equation*}
\sum_{\substack{x,y \in \widetilde \cX \\ \{ i \} \subsetneq D(x) \cap D(y)}} \one_{x|_i^{} = x^i} \one_{y|_i^{} = y^i} \nu (x) (\nu - \delta_x) (y).
\end{equation*}
It is easy to see that this means that transitions
\begin{equation*}
(\nu,\sigma(\nu),0) \to (\nu, C_{x|_i^{},y|_i^{}} (\sigma(\nu)) , 1)
\end{equation*}
have to occur at rate $\nu(x) (\nu - \delta_x) (y)$, independently for all $x,y \in \widetilde \cX$ with $|D(x) \cap D(y)| \geqslant 2$ and all $i \in D(x) \cap D(y)$.

Recombination events in $\widetilde \cR$ do not lead to transitions in $\widetilde \cR^\infty$. Assume that a recombination event occurs in $\widetilde \cR$ at time $t$, i.e. $\widetilde \cR_t^{} = \widetilde \cR_{t-} - \delta_x^{} + \sum_{A \in \cA} \delta_{x|_A^{}}$ for some 
$x \in \widetilde \cX$ and some $\cA \in \bP(S)$. Then, we have 
\begin{equation} \label{invarianceunderreco}
\begin{split}
\sigma(\widetilde \cR_t^{}) 
& = \sigma(\widetilde \cR_{t-}^{} ) - \sigma(\delta_x) + \sum_{A \in \cA} \sigma ( \delta_{x|_A^{}}   ) \\ 
& =  \sigma(\widetilde \cR_{t-}^{} )  - \sum_{i \in D(x)} \delta_{x|_i^{}} + \sum_{A \in \cA} \sum_{i \in D(x) \cap A} \delta_{x|_i^{}} \\ 
& =  \sigma(\widetilde \cR_{t-}^{} ),
\end{split}
\end{equation}
where the sums cancel in the last step because $\cA$ is a partition of S. 
So, simply letting $\widetilde \cR^\infty_t = \widetilde \cR^\infty_{t-}$ does the trick of preserving $\widetilde \cR^\infty_t = \sigma ( \widetilde \cR_t^{} )$. Matching mutation events is trivial since mutations occur on each site independently.
We summarise all of this by giving a formal Definition.

\begin{definition} \label{def:cmarg}
We let ($\widetilde \cR, \widetilde \cR^\infty, \fm) = \big ( (\widetilde \cR_t^{}, \widetilde \cR_t^\infty, \fm_t^{})  \big )_{t \geqslant 0}$ be a continuous-time Markov chain on
 $(\cN(\widetilde \cX) \cup \{\Delta\}) \times (\cN(\widetilde \cX_1^{}) \cup \{\Delta \} ) \times \{0,1\}$ 
with $\widetilde \cX_1^{}$ as in Eq.~\eqref{x1definition}. When $(\widetilde \cR, \widetilde \cR^\infty, \fm)$ is in a state of the form $(\nu, \mu, 0)$ for some
$\nu \in \widetilde \cX$ and $\mu \in \widetilde \cX_1^{}$,
we say that
 $\widetilde \cR$ and 
$\widetilde \cR^\infty$ are unmatched and they perform independent transitions according to Defs.~\ref{def:measureARG} and \ref{def:sm-arg}, while $\fm$ remains at $0$ for all time. While $\fm = 1$ and $\mu = \sigma(\nu)$, we say that $\widetilde \cR$ and $\widetilde \cR^\infty$ are matched and
$(\widetilde \cR, \widetilde \cR^\infty_{}, \fm)$ performs the following independent transitions.
\begin{enumerate}
\item
\textbf{Good coalescence}: Independently for all $x,y \in \widetilde \cX$ with \mbox{$|D(x) \cap D(y)| \leqslant 1$},
\begin{equation*}
(\nu,\sigma(\nu),1) \to (C_{x,y} (\nu), \sigma ( C_{x,y} (\nu)  ), 1)  
\end{equation*}
at rate $\nu(x) (\nu - \delta_x) (y)$.
\item
\textbf{Bad coalescence (in $\widetilde \cR$)} : Independently for all $x,y \in \widetilde \cX$ with $|D(x) \cap D(y)| \geqslant 2$,
\begin{equation*}
(\nu,\sigma(\nu),1) \to (C_{x,y} (\nu), \sigma(\nu),0)
\end{equation*}
at rate $\nu(x) (\nu - \delta_x)(y)$.
\item
\textbf{Bad coalescence (in $\widetilde \cR^{\infty}$)}: Independently for all $x,y \in \widetilde \cX$ with $|D(x) \cap D(y) | \geqslant 2$ and each $i \in D(x) \cap D(y)$,
\begin{equation*}
(\nu,\sigma(\nu),1) \to (\nu,C_{x|_i^{},y|_i^{}}(\sigma(\nu)),0)
\end{equation*}
at rate $\nu(x) \nu(y)$ if $x \neq y$ and at rate $\nu(x) ( \nu(x) - 1 )$ if $x = y$. 
\item
\textbf{Recombination:} Independently for all $\cA \in \bP(S)$ and all $x \in \widetilde \cX$,
\begin{equation*}
(\nu, \sigma(\nu),1) \to \Big ( \nu - \delta_x^{} 
+ \sum_{\substack{A \in \cA \\ A \cap D(x) \neq \varnothing}} \delta_{x|^{}_A}^{}, \sigma(\nu), 1 \Big )
\end{equation*}
at rate $\varrho_\cA^{} \nu(x)$.
\item
\textbf{Mutation:} Independently for each $x \in \widetilde \cX$, $i \in D(x)$ and $y_i^{},z_i^{} \in X_i$,
\begin{equation*}
(\nu,\sigma(\nu),1) \to
\begin{cases}
 ( \nu - \delta_x^{} + \delta_{m(x;y_i^{},z_i^{})}^{}, \sigma(\nu) - \delta_{x|_i^{}}^{} + \delta_{m(x;y_i^{},z_i^{})|_i^{}}^{},1) 
& \textnormal{if } m(x;y_i^{},z_i^{})(i) \neq \varnothing, \\
(\Delta,\Delta,1) & \textnormal{if } m(x;y_i^{},z_i^{})(i)= \varnothing
\end{cases}
\end{equation*}
at rate $u_i^{} M_i(y_i^{}, z_i^{}) \nu (x)$ with $m(x;y_i^{},z_i^{})$ as in Eq.~\eqref{mutantcandidates}.
\end{enumerate}
In addition, we define the \emph{unmatching time} (of $\widetilde \cR$ and $\widetilde \cR^\infty$) 
\begin{equation*}
T_u \defeq \min \{ t \geqslant 0 : \fm_t = 0 \}.
\end{equation*}
\end{definition}

Note that when $\fm_t^{} = 1$, $\widetilde \cR_t^{}$ is simple if and only if $\widetilde \cR_t^\infty$ is. In that case, no bad coalescence can occur, i.e.
we have $\fm_s^{} = 1$ for all $s > 0$.

\begin{remark}\label{rmk:asparticlesystems}
There is an obvious way to embed the measure-valued processes $\widetilde \cR$ and $\widetilde \cR^\infty$ within coupled, coalescing and fragmenting, typed particle systems. 
Starting $\widetilde \cR$ from $\nu$ and $\widetilde \cR^\infty$ from $\sigma(\nu)$, we 
start the corresponding particle systems with $\nu(x)$ particles of each type $x \in  \widetilde \cX$ and $\sigma(\nu)(x^i) = \nu^{\{i\}}(x^i)$ particles of type
$x^i \in \widetilde X_{ \{i \}}^{}$ for each $i \in S$, respectively. 

Now, 1 (and 2) in Definition~\ref{def:cmarg} imply that each ordered pair of particles with types $x$ and $y$ in $\widetilde \cR$ such that $|D(x) \cap D(y)| = 1$ ($\geqslant 2$)  independently performs a good (bad) coalescence at rate $1$, in line with the rate for such events being $\nu(x) (\nu - \delta_x^{})(y)$. Moreover, 3 implies that any ordered pair of particles in $\widetilde \cR$ with types $x$ and $y$, independently and independently for each $i \in D(x) \cap D(y)$ triggers a bad coalescence in $\widetilde \cR^\infty$ between a randomly chosen pair of particles of type $x|_i^{} \in X_i^{}$ and $y|_i^{} \in X_{ \{ i \}}^{}$. 
Furthermore, 4 implies that each particle in $\widetilde \cR$ with, say, type $x$, is, independently at rate $\varrho_\cA^{}$, fragmented into fragments of
types $x|_A^{}$ for each $A \in \cA|_{D(x)}^{}$. Finally, (5) implies that mutation happens for each particle independently and at each site $i$ at rate $u_i^{}$ and the allele change is governed by the mutation kernel $M_i^{}$; here, every particle in $\widetilde \cR$ is matched to its single-site fragments in $\widetilde \cR^\infty$. 

Because of permutation invariance, the concrete particles that are involved in each event can simply be chosen randomly and thus, the law of the particle systems are uniquely determined by the evolution of the frequency processes $\widetilde \cR$ and $\widetilde \cR^\infty$. This argument was also used in~\cite{JenkinsFearnheadSong2015} who expressed the law of the untyped ARG (for two loci) in terms of its line counting process which they termed the \emph{ancestral process}.

To make this formally rigorous, one could, for instance, assign a uniform (on $[0,1]$) label to each point mass, i.e. consider the cmARG as a process on subsets of $\widetilde \cX$ and work with relations between particles at different times; we refer the curious reader to~\cite{GrevenPfaffelhuberPokalyukWakolbinger2016} where a similar construction has been carried out for the ancestral selection graph~\cite{KroneNeuhauser97}.
\end{remark}

We have constructed $(\widetilde \cR, \widetilde \cR^\infty, \fm)$ such that the marginal law of $\widetilde \cR$ is that of a mARG, and the marginal law of 
$\widetilde \cR^\infty$ is that of an s-mARG. For completeness, we briefly summarise the arguments from the beginning of this section to prove the following 
\begin{lemma} \label{lem:iscoupling}
Let $\nu \in \cN(\widetilde \cX)$ and let $(\widetilde \cR, \widetilde \cR^\infty, \fm)$ be as in Definition~\ref{def:cmarg}, started from $(\nu,\sigma(\nu),1)$. Then, $\widetilde \cR$ and $\widetilde \cR^\infty$ are a mARG and an s-mARG, started from $\nu$ and $\sigma(\nu)$, respectively.
\end{lemma}
\begin{proof}
We see that taking (1) and (2) from Def.~\ref{def:cmarg} together and considering only the first component gives the transitions in (1) from Def.~\ref{def:measureARG}. The left component of (4) in Def.~\ref{def:cmarg} is precisely (2) in Def.~\ref{def:measureARG}. The left component of the transition (5) in Def.~\ref{def:cmarg} is just (3) in Def.~\ref{def:measureARG}. This shows that $\widetilde \cR$ is a mARG.

To see that $\widetilde \cR^\infty$ is an s-mARG, note that, by (5) of Def.~\ref{def:cmarg}, transitions 
\begin{equation*}
\sigma(\nu) \to 
\begin{cases}
\sigma(\nu) - \delta_{x^i}^{} + \delta_{m(x^i;y_i^{},z_i^{})}^{} & \textnormal{if } m(x_i^{};y_i^{},z_i^{}) \neq \varnothing ,\\
\Delta & \textnormal{if } m(x^i;y_i^{},z_i^{})  = \varnothing,
\end{cases}
\end{equation*}
occur independently for different $x^i \in X_{\{i\}}$ (as well as different $y_i^{}$, $z_i^{}$) because $x \neq y$ when $x|_i^{} \neq y|_i^{}$. The total transition rate is given by
\begin{equation*}
u_i^{} M_i^{}(y_i^{},z_i^{}) \sum_{x \in \widetilde \cX} \nu(x) \one_{x|_i^{} = x^i}^{}(x) =  u_i^{} M_i^{}(y_i^{},z_i^{})  \sigma(\nu) (x^i) = u_i^{} M_i^{}(y_i^{},z_i^{})  \sigma(\nu)(x^i),
\end{equation*}
in line with 2 in Def.~\ref{def:sm-arg}.

For $x^i \in \widetilde X_{\{ i \}}$ and $y^i \in \widetilde X_{ \{ i \}}^{}$, transitions $\sigma(\nu) \to C_{x^i, y^i} (\sigma(\nu))$ happen independently, at rate 
\begin{equation*}
\sum_{\substack{x,y \in \widetilde \cX \\ D(x) \cap D(y) = \{i\} }} \one_{x|_i^{} = x^i}^{}  \one_{y|_i^{} = y^i}^{} \nu(x) (\nu - \delta_x) (y)
\end{equation*}
due to good coalescence and at rate 
\begin{equation*}
\sum_{\substack{x,y \in \widetilde \cX \\ \{ i \} \subsetneq D(x) \cap D(y)}} \one_{x|_i^{} = x^i} \one_{y|_i^{} = y^i} \nu (x) (\nu - \delta_x) (y)
\end{equation*}
due to bad coalescence.
So, the rate is 
\begin{equation*}
\sum_{\substack{x,y \in \widetilde \cX \\  i \in  D(x) \cap D(y)}} \one_{x|_i^{} = x^i} \one_{y|_i^{} = y^i} \nu (x) (\nu - \delta_x) (y),
\end{equation*}
which we have already seen (cf. Eq.~\eqref{soll} and the preceding calculation) to be equal to 
$\sigma(\nu) (x^i) \big ( \sigma(\nu) - \delta_{x^i}  \big )(y^i)$, 
in line with 1 in Definition~\ref{def:sm-arg}.
\end{proof}

In what follows, we will write $\PP_\nu^{}$ for the law of $(\widetilde \cR, \widetilde \cR^\infty,\fm)$, started from $(\nu,\sigma(\nu),1)$, and $\EE_\nu$ for the corresponding expectation. We will also drop the tilde and work throughout with versions of $\cR$ and $\cR^\infty$ that are coupled in this way.

Our next Lemma shows that for large $\varrho$, the probability that $\cR$ and $\cR^\infty$ unmatch due to bad coalescence is of order 
$\varrho^{-1}$.  In fact, we know a little bit more: with high probability, the bad coalescence that led to the unmatching is the first event apart from recombination, i.e., it occurs prior to any other coalescence or mutation event. It is this observation that will make the asymptotic analysis of
$(\cR,\cR^\infty)$ and therefore of $q(\nu)$ tractable. 

In particular, this means that, with high probability $\cR^\infty_t = \cR^\infty_0$ (or, equivalently, $\sigma (\cR_t^{}) = \sigma(\cR_0^{})$ due to $\cR$ and $\cR^\infty$ being matched)
for all $t$ with
$\fm_t^{} = 1$. At the same time, there also being no coalescence between particles observed at disjoint sites (which would be compatible with $\cR^\infty_t = \cR_t^{}$ for all $t < T_u^{}$); more formally, this means that $\cR_s (\widetilde \cX) \leqslant \cR_t (\widetilde \cX)$ 
for all $0 \leqslant t < T_u$, i.e. the total number of particles is nondecreasing as long as $\fm_t^{} = 1$.

\begin{lemma}\label{lem:workhorse}
Assuming $( \cR,  \cR^\infty)$ to be coupled as in Definition~\ref{def:cmarg}, we define the following events.
\begin{equation*}
E \defeq \{ T_u = \infty  \} = \{ \fm_t^{} = 1 \textnormal{ for all } t \geqslant 0   \},
\end{equation*}
%\todo{can we drop this additional stopping time?}
%where $\widetilde T \defeq \max(T,T^\infty)$ is the smallest time $t$ such that both $\cR_t^{}$ and $\cR_t^\infty$ are either simple or in the cemetary. Note %that since $\cR$ and $\cR^\infty$ decouple and evolve independently of each other after the occurrence of a bad coalescence, it might happen that
%$\cR_{\widetilde T}^{}$ is in the cemetary while $\cR_{\widetilde T}^\infty$ is simple, or vice versa. 
and
%\begin{equation*}
%F \defeq \{ \textnormal{The first transition apart from recombination is a bad coalescence} \}.
%\end{equation*}
\begin{equation*}
F \defeq \{ T_u < \infty, 
\cR_t^\infty = \cR_0^\infty  \textnormal{ for all } t < T_u  \textnormal{ and } \cR_s^{} (\widetilde \cX) \leqslant \cR_t^{} (\widetilde \cX)
\textnormal{ for all } 0 \leqslant s \leqslant t < T_u  \}.
\end{equation*}
Note the strict inequality in the definition of F.
Then, the following holds.
\begin{enumerate}
%\item
%On $E$, we have $\widetilde T = T^\infty = T$ and $\sigma(\cR_t^{}) = \cR^\infty_t$ for all $t \in [0,T]$
\item
$\PP_\nu(E) = 1 - \cO(\varrho^{-1})$.
\item
$\PP_\nu(E \cup F) = 1 - \cO(\varrho^{-2})$.
\end{enumerate}
\end{lemma}

\begin{proof}
%The statement (1) is easily shown by applying $\sigma$ to the left component of the state after each transitions in Definition~\ref{def:cmarg}.

%To show (2) and (3), 
Since bad coalescence can only occur at times when $\cR$ gives mass to types $x$ with $|D(x)| \geqslant 2$, we introduce two sequences of stopping times, $t_1^{},t_2^{},\ldots,$ and $s_1^{},s_2^{},\ldots$. Let
$s_1^{} \defeq 0$. We call a particle a \emph{singleton} if its type is observed at a single site only and define inductively for all $i \geqslant 1$
\begin{equation*}
\tau_i^{} \defeq \inf \{ s > s_i^{} : \cR_s^{} (x) = 0 \textnormal{ for all } x \in \widetilde \cX \textnormal{ with } |D(x)| \geqslant 2  \},
\end{equation*} 
and
\begin{equation*}
s_{i+1}^{} \defeq \inf \{ s > t_i^{} : \textnormal{there are } x \in \widetilde \cX \textnormal{ with } |D(x)| \geqslant 2 \textnormal{ and }\cR_s^{} (x) > 0 \}.
\end{equation*}
Then, bad coalescence can only occur in the open intervals $(s_i^{}, t_i^{})$. 
The claims will follow from the following two facts, which will be proven below.
\begin{enumerate}[label=(\alph*)]
\item
For all $i$, the probability that at least $k$ events occur in $(s_i^{}, t_i^{})$ that are not recombinations is bounded by
$\cO(\varrho^{-k})$.
\item
The total number of such intervals before $\cR$ becomes simple has a subexponential tail, i.e.
$
G \defeq \max \{i \in \NN : \cR_{t_i^{}} \textnormal{ is not simple}\}
$
satisfies
$
\PP_\nu (G \geqslant t)  \leqslant C \gamma^{t}
$
for some constant $C$ and $\gamma \in (0,1)$.
\end{enumerate}

Let's first see how (a) and (b) together imply the claims. For 1, note that (b) implies the existence of a deterministic function $\varrho \mapsto T_\varrho^{} = \cO(\log(\varrho^{-2}))$ such that 
\mbox{$\PP_\nu ( G \geqslant T_\varrho) \leqslant \varrho^{-2}$}.
Since $\cR$ will stay simple forever once it becomes simple (and we will then never see a bad coalescence), any bad coalescence must occur during 
$(s_i^{}, t_i^{})$ for some $1 \leqslant i \leqslant G$. Moreover, for a bad coalescence to occur in $(s_i^{}, t_i^{})$ for 
$i \geqslant 2$, at least one additional coalescence event must occur in that time-interval prior to the bad coalescence; this is because at time $s_i^{}$, there is only one
$x^\ast$ with $|D(x)^\ast| \geqslant 2$ with $\cR_s (x^\ast) = 1$, and for all $x \neq x^\ast$ with $|D(x)| \geqslant 2$, $\cR_s (x) = 0$ (only one non-singleton type is produced by the coalescence at time $s_i^{}$). Thus, we can utilise a simple union bound and see
\begin{equation*}
\begin{split}
\PP_\nu ( T_u < \infty)  
& \leqslant \PP_\nu ( T_u \in (s_i^{}, t_i^{}) \textnormal{ for some } 1 \leqslant i \leqslant G   ) \\
&  \leqslant \PP_\nu ( T_u \in (s_i^{}, t_i^{}) \textnormal{ for some } 2 \leqslant i \leqslant G   ) + \cO(\varrho^{-1}) \\
& \leqslant \PP_\nu ( T_u \in (s_i^{}, t_i^{}) \textnormal{ for some } 2 \leqslant i \leqslant T_\varrho   ) + \cO(\varrho^{-1}) \\
&  \leqslant T_\varrho \cO (\varrho^{-2}) + \cO(\varrho^{-1})  \\
&  = \cO(\varrho^{-1}).
\end{split}
\end{equation*}

For 2, we need to show that 
\begin{equation*}
\PP_\nu ( T_{\neq} < T_u < \infty ) = \cO(\varrho^{-2}).
\end{equation*}
where $T_{\neq} \defeq \min \{ t \geqslant 0 : \sigma ( \cR_t ) \neq \sigma ( \cR_0  ) \}$.
Recall from Eq.~\eqref{invarianceunderreco} that $\sigma(\cR_t^{})$ is invariant under recombination. Thus, in order for the event inside this probability to occur, we need to observe at least one transition prior to $T_u$ that does not correspond to recombination. More precisely, when 
$T_u \in (s_1^{}, t_1^{})$, we would need to observe an addtional non-recombination event, so a total of two non-recombination events.  On the other hand, if $T_u \in ( s_i^{}, t_i^{}   )$ for $i \geqslant 2$, we need to (as in the proof of 1) observe two coalescence events in $(s_i^{}, t_i^{})$, including the bad one. However, the first of these two also leaves $\sigma(\cR)$ invariant, as it has to be a coalescence between two singleton types. Thus, a \emph{third} non-recombination event (either a mutation or another coalescence needs to be observed). Thus, a union bound analogous to the one we saw above yields
\begin{equation*}
\PP_\nu ( T_{\neq } < T_u < \infty   ) \leqslant T_\varrho \cO(\varrho^{-3}) + \cO(\varrho^{-2}) = \cO(\varrho^{-2}).
\end{equation*}
Now, we show (a) and (b) by comparing the coalescent to a random walk. 

\emph{Proof of (a):} Note that on $(s_i^{}, t_i^{})$, the total rate of \emph{nonsilent} recombination is bounded from below by
$c \varrho$ where $c \defeq \min \{r_\cA^{} : \cA \in \bP(S), r_\cA^{} \neq 0 \}$ . At the same time, the total rate of events that are not recombinations is bounded from above by a uniform constant $C$, because the total number of particles remains bounded at all times. Also note that any coalescence increases the possible number of consecutive nontrivial fragmentations by not more than $n = |S|$. Thus, the number of nonfragmentation events in $(s_i^{}, t_i^{})$ events is stochastically dominated by the number of upward steps performed by a random walk on $\ZZ$, started from some $M > 0$ and before hitting $\ZZ_{\leqslant 0}$, with homogeneous transition probabilities
\begin{equation*}
p(z,z-1) = \frac{\varrho c}{\varrho c + C}, \quad p(z,z+n) = \frac{C}{\varrho c + C}.
\end{equation*}
Lemma~\ref{lem:excursions} then yields the desired bound.

\emph{Proof of (b):} If $\nu = \cR_t^{}$ is not simple, then either there are subsets $A$ and $B$ of $S$ with $\nu(\widetilde X_A^{}) > 0$, $\nu(\widetilde X_B^{}) > 0$ and 
$A \cap B \neq \varnothing$. Or, there is an $A \subseteq S$ with $\nu (\widetilde X_A^{}) \geqslant 2$.
In the former case, a coalescence between particles of (compatible) types $x_A^{} \in \widetilde X_A^{}$ and $x_B^{} \in \widetilde X_B^{}$ leads to a decrease of the total masses of $\widetilde X_A^{}$ and $\widetilde X_B^{}$ (while increasing that of $\widetilde X_{A \cup B}^{}$). In the latter case, a coalescence between two compatible particles with types $x,y \in \widetilde X_A^{}$ decreases the mass of $\widetilde X_A^{}$. In any case, we call such a coalescence \emph{reducing}.

Now, if $\cR_{t_i^{}}^{}$ is not simple, there is at least one possible reducing coalescence. Also, the total coalescence rate can be bounded by the same constant $C$ used earlier. Therefore, the coalescence at time $s_{i+1}^{}$ is a reducing coalescence with probability at least $1/C$. And because $\cR$ can only perform a finite number of reducing coalescences before becoming simple or hitting $\Delta$, we see that $G$ is stochastically dominated by a random variable with negative binomial distribution. In particular, $G$ has a subexponential tail.
\end{proof}

\begin{remark}\label{rem:uniform}
The proof of Lemma~\ref{lem:workhorse} also shows that for any $M > 0$, the constant implied by the $\cO$ is uniform for all $\nu$ with $\| \nu \| \leqslant M$.
\end{remark}

By Eq.\eqref{nonsimplesample} and Lemma~\ref{lem:workhorse} (2), we have
\begin{equation*}
q (\nu) = \EE_\nu \big [ q ( \cR_T   )   \big ] = \EE_\nu \big [ q_\infty^{} ( \sigma(\cR^{}_T   )) , E  \big ] + \EE_\nu \big [ q_\infty^{} ( \sigma (\cR_T )  ) , F  \big ] + \cO(\varrho^{-2}),
\end{equation*} 
where we used Lemma~\ref{lem:qvsqinfty} and that $\cR_T$ is by definition of $T$ either simple or $\Delta$.
On $E$, we have $\fm_t = 0$ for all $t \geqslant 0$, meaning that $\cR_t^\infty = \sigma (\cR_t^{})$ for all $t \geqslant 0$ and $T^\infty = T$ because
clearly, $ \sigma (\cR_t^{})$ is simple if and only if $\cR_t^{}$ is.
 Therefore, the right-hand side is equal to
\begin{equation} \label{aewkfj}
\EE_\nu \big [ q ( \cR_T   )   \big ] = \EE_\nu \big [ q_\infty^{} ( \sigma( \cR^\infty_{T^\infty})   ) , E  \big ] + \EE_\nu \big [ q_\infty^{} ( \sigma (\cR_T )  ) , F  \big ] + \cO(\varrho^{-2}).
\end{equation}
Similarly, we can use Eq.~\eqref{nonsimplesampleq} and Lemma~\ref{lem:workhorse} (2) to see that
\begin{equation*}
q_\infty^{} (\nu) = q_\infty^{} (\sigma(\nu)) = \EE_\nu \big [  q_\infty^{} (\sigma (\cR^\infty_{T^\infty})), E  \big ] + \EE_\nu \big [  q_\infty^{}(\sigma(\cR^\infty_{T^\infty})), F          \big ] + \cO(\varrho^{-2}).
\end{equation*}
Solving for $\EE_\nu \big [  q_\infty^{}(\sigma(\cR^\infty_{T^\infty})), E          \big ]  $ and inserting the result into Eq.~\eqref{aewkfj}, we arrive at
\begin{equation*}
q(\nu) = q_\infty^{} \big ( \sigma(\nu)   \big )  + \EE_\nu \big [ q_\infty^{}( \sigma(\cR^{}_T)), F \big ] - \EE_\nu \big [ q_\infty^{}( \sigma(\cR^\infty_{T^\infty})), F  \big ] +  \cO (\varrho^{-2}).
\end{equation*}
Comparing this with Eq.~\eqref{refinedapproximation}, it is clear that to get a handle on the first-order term in Eq.~\eqref{samplingformula}, we need to evaluate the difference
\begin{equation}\label{q1asexpectation}
\EE_\nu \big [ q_\infty^{}(\sigma(\cR_T^{})), F \big ] - \EE_\nu \big [ q_\infty^{}(\sigma(\cR^\infty_{T^\infty})),  F  \big ].
\end{equation}
Note that this is in fact of order $\cO(\varrho^{-1})$ because $F$ is a subset of the complement of $E$ and $\PP_\nu (E) = 1 - \cO(\varrho^{-1})$ by 1 of
Lemma~\ref{lem:workhorse}.

Recall that on $F$, we have in particular that $T_u^{} < \infty$. Thus, we can disintegrate the expectation with respect to 
$\cR_{T_u^{}}$, i.e.
\begin{equation*}
\begin{split}
\EE_\nu \big [ q_\infty^{} (\sigma(\cR_T^{})), F    \big ] 
& = \EE_\nu \big [ \EE_{\cR_{T_u}^{}} [ q_\infty^{} (\sigma(\cR_T^{}) ) ], F   \big ] \\
& = \EE_\nu \big [ q(\cR_{T_u}^{})   ,F        \big ]  \\
& =  \EE_\nu \big [ q_\infty(\sigma(\cR_{T_u}^{})),F \big ] + \cO(\varrho^{-2}).
\end{split} 
\end{equation*}
Note that the last step uses that $q(\cR_{T_u}^{} ) =  q_\infty(\sigma(\cR_{T_u}^{})) + \cO(\varrho^{-1})$ 
(locally) uniformly in the argument and $\PP_\nu (F) = \cO(\varrho^{-1})$.
Similarly,
\begin{equation*}
\EE_\nu \big [ q_\infty^{}(\sigma(\cR^\infty_{T^\infty})),  F  \big ] =  \EE_\nu \big [ q_\infty^{} (\sigma(\cR_{T_u}^\infty))    ,F       \big ];
\end{equation*}
here, the $\sigma$ is only there for cosmetic reasons because $\cR^\infty$ is supported on singletons. 

To summarise, we can (up to a negligible error of order $\varrho^{-2}$) identify $\varrho^{-1} q_1^{} (\nu)$ in Eq.~\eqref{samplingformula} with
\begin{equation} \label{mainthingtoevaluate}
\EE_\nu \big [ q_\infty(\sigma(\cR_{T_u}^{})),F \big ]  - \EE_\nu \big [ q_\infty^{} (\sigma(\cR_{T_u}^\infty))    ,F  \big ]
\end{equation}

Because we have on $F$ that $T_u^{} < \infty$ while $\sigma (\cR_t^{}) = \sigma ( \cR_0^{})$ for all 
$t < T_u$, $\sigma(\cR_{T_u}^{})$ and $\cR_{T_u}^\infty$ only depend on
the bad coalescence that occurred at time $T_u$. For instance, if 
\begin{equation*}
\cR_{T_u}^{} = \cR_{T_u -}^{} - \delta_x - \delta_y + \delta_{x \sqcup y}
\end{equation*}
for some (compatible) $x, y \in \cX$ with $|D(x) \cap D(y)| \geqslant 2$ (that is, we see a bad coalescence in $\cR$), then we have
\begin{equation*}
\begin{split}
\sigma(\cR_{T_u}^{}) \\ 
& \quad = \sigma(\cR_{T_u - }^{}) - \sum_{i \in D(x)} \delta_{x|_i^{}} - \sum_{i \in D(y)} \delta_{y|_i^{}} + \sum_{i \in D(x) \cup D(y)} \delta_{(x \sqcup y)|_i^{}} \\
& \quad = \sigma(\cR_0^{}) - \sum_{i \in D(x) \cap D(y)} \delta_{x|^{}_i} \\
& \quad = \sigma(\cR_0^{}) - \sigma(\delta_{x|_{D(x) \cap D(y)}^{}}).  
\end{split}
\end{equation*}
Of course, if the types are incompatible, we have $q_\infty^{} (\cR_T^\infty) = 0$, so such events do not contribute to Eq.~\eqref{q1asexpectation}. Put more succinctly, $q_\infty^{} (\sigma (\cR_T^\infty ))$ is either $0$ or of the form \mbox{$q_\infty^{} (\sigma(\cR_0^{}) - \sigma(\delta_z))$} for some $z \in \widetilde X$ with $|D(z)| \geqslant 2$. At the same time, because a bad coalescence in $\cR$ is not matched by a transition in $\cR^\infty$,  we have
$q_\infty^{} (\sigma (\cR_{T_u}^\infty)) = q_\infty^{} (\sigma(\cR_0^\infty))$.

Similarly, if at time $T_u$ we observe a bad coalescence in $\cR^\infty$ , we have 
$
q_\infty^{} ( \sigma( \cR_{T_u}^{} )    ) = q_\infty^{} (\sigma (\cR_0^{} ) )
$
while, for some $i \in S$ and $x \in X_{ \{ i \}  }$, 
$
q_\infty^{} ( \sigma( \cR_{T_u}^\infty )    ) =  q_\infty^{} (\sigma (\cR_0^\infty ) - \delta_x ) = q_\infty^{} (\sigma (\cR_0^\infty ) - \sigma(\delta_x )),
$
or the right-hand side vanishes if the coalescence was between incompatible types.
\begin{remark}
Note that the types $x$ are, in fact, elements of $\cX$, i.e. exact types. This is because we start $\cR$ and $\cR^\infty$ from $\nu$ supported on exact types, and on the event $F$, no fuzzy types arise prior to $T_u$.
\end{remark}

From this, we see that the difference of expectations in Eq.~\eqref{mainthingtoevaluate} is of the form
\begin{equation} \label{roughsamplingformula}
\sum_{x \in \cX} q_\infty^{} \big ( \sigma(\nu) - \sigma(\delta_x)    \big ) \cdot K(x),
\end{equation}
with some coefficients $K(x)$, depending on $\varrho$. To determine them, note that 2 and 3 in Defintion~\ref{def:cmarg} imply the following Poissonian construction of bad coalescence events. Start by letting $(\cR,\cR^\infty)$ evolve, but at first only taking into account transitions in 1, 4, and 5 (good coalescence, recombination, and mutation). Alongside this, for all $x,y \in \widetilde \cX$ with $|D(x) \cap D(y)| \geqslant 2$, run independent Poisson point processes $\pi_{x,y}^{}$ on 
$\RR_{\geqslant 0}$ with inhomogeneous rates $\cR_t (x) (\cR_t - \delta_x) (y) (|D(x) \cap D(y)| + 1)$. When we encounter an atom $t$ of $\pi_{x,y}^{}$, we say that, at time $t$, a bad coalescence is triggered by the pair $(x,y)$. In that case, we choose an element of $\{ 0 \} \cup (D(x) \cap D(y))$ uniformly. If we choose $0$, we say that \emph{bad coalescence is triggered by the pair $(x,y)$ in $\cR$}, and let
$(\cR_t^{}, \cR_t^\infty) = (C_{x,y} (\cR_{t-}^{}), \cR_{t-}^\infty)$. If we choose $i \in D(x) \cap D(y)$, we say that \emph{bad coalescence is triggered by the pair $(x,y)$ at site $i$} and let
$(\cR_t^{}, \cR_t^\infty) = (\cR_{t-}^{}, C_{x|_i^{},y|_i^{}}(\cR_{t-}^\infty))$. Note that the factor $(|D(x) \cap D(y)| + 1)$ in the intensity of $\pi_{x,y}^{}$ is chosen such that each of these transitions is independently triggered at rate $\cR_t (x) (\cR_t - \delta_x) (y)$, as in 2 and 3 of Definition~\ref{def:cmarg}.

Next, we express the coefficients $K(x)$ in Eq.~\eqref{roughsamplingformula} in terms of the probabilities
\begin{equation*}
p_{x,y}^{} \defeq \PP_\nu (F \cap \{ \textnormal{The bad coalescence at time } T_u \textnormal{ is triggered by } (x,y) \textnormal{ in } \cR \}   ).
\end{equation*}
Note that $p_{x,y}^{}$ is also the probability that bad coalescence is triggered by $(x,y)$ at site $i$ for any $i \in D(x) \cap D(y)$. Also keep in mind that on $F$, there are no mutations prior to $T_u$ and hence the bad coalescence has to be triggered by a pair of exact types. 

We start with $K(\epsilon)$, the coefficient in front of $q_\infty^{} \big ( \sigma(\nu)  \big )$. Eq.~\eqref{mainthingtoevaluate} and the discussion below it shows that 
\begin{equation*}
K(\epsilon) =  \sum_{\substack{x,y \in \cX \\ |D(x) \cap D(y)| \geqslant 2}} p_{x,y}^{} (|D(x) \cap D(y)| - 1);
\end{equation*}
for each $x,y$, we get a positive contribution from the first expectation due to $(x,y)$ inducing bad coalescence  at each site $i \in |D(x) \cap D(y)|$, and a negative contribution due to $(x,y)$ inducing bad coalescence in $\cR$.
For $z \in \cX$ with $|D(z)| = 1$, we only get negative contributions from the second expectation, due to $(x,y)$ inducing bad coalescence at the unique site $i \in D(z)$. So,
\begin{equation*}
K(z) = - \sum_{\substack{x,y \in \cX \\ D(z) \subsetneq D(x) \cap D(y)}} p_{x,y}^{} \one_{x|_{D(z)}^{} = y|_{D(z)}^{} = z}.
\end{equation*}
Here, $D(z)$ has to be a proper subset because it has cardinality one.
Finally, if $z \in \cX$ with $|D(z)| \geqslant 2$, we only get positive contributions from the first expectation, due to bad coalescence induced in $\cR$.
\begin{equation*}
K(z) = \sum_{\substack{x,y \in \cX \\ D(z) = D(x) \cap D(y)}}  p_{x,y}^{} \one_{x|_{D(z)}^{} = y|_{D(z)}^{} = z}.
\end{equation*}
These three sums can be evaluated via an inclusion-exclusion principle, which is the content of the next and final section.

\section{Taming bad coalesence} \label{sec:taming}
In order to streamline our argument we need one additional piece of notation
\begin{definition}
For all $z \in \cX$ and $A \subseteq S$, we write $P(A,z)$ for the probability that a  pair $(x,y) \in \cX^2$ with 
$D(x) \cap D(y) = A$ and $x|_{D(z)}^{} = y|_{D(z)}^{} = z$ triggers a bad coalescence in $\cR$ (or, equivalently, a  bad coalescence in $\cR^\infty$ at site $i$ for any $i \in A$), prior to any other coalescence or mutation event.
\end{definition}
%This definition only explicitly deals with bad coalescence in $\cR$. But note that for any $i \in A$, $P(A,z)$ is also the probability that a bad %coalescence is triggered by this pair at site $i$ in $\cR^\infty$; see the definition of $p_{x,y}^{}$ above and the surrounding discussion.  

Keep in mind that because we are interested in the probabilities of events that happen before mutation, we will be working with \emph{exact} types (elements of $\cX$) rather than fuzzy types (elements of
$\widetilde \cX$). We will also be suppressing the dependence on the initial condition $\nu$.

With this, we can rewrite the three sums from the end of last section as follows: We have
\begin{equation} \label{zerosites}
K(\epsilon) = \sum_{\substack{ A \subseteq S \\ |A| \geqslant 2}} P(A,\epsilon) \big ( |A| -1     \big ),
\end{equation}
for $z \in \cX$ with $|D(z)| = 1$ we have
\begin{equation} \label{onesite}
K(z) = - \sum_{D(z) \subsetneq A \subseteq S} P(A,z),
\end{equation} 
and for $|D(z)| \geqslant 2$ we have 
\begin{equation} \label{atleasttwosites}
K(z) = P(D(z),z).
\end{equation}
The main task is to evaluate $P(A,z)$. For this, we will need the following Lemma, which is a consequence of~\cite[Chapter IV, 4.18]{Aigner79} for the partially ordered set $(2^S, \subseteq)$ and the field of real numbers, where $2^S$ is the powerset of $S$.
\begin{lemma} \label{lem:exclincl}
Let $g : 2^S \to \RR$. Setting
\begin{equation*}
G(A) \defeq \sum_{\udo{B} \supseteq A} g(B),
\end{equation*}
$g$ can be recovered from $G$ via
\begin{equation*}
g(A) = \sum_{\udo B \supseteq A}  (-1)^{|B \setminus A|} G(B).
\end{equation*}
Here the underdot marks the summation variable. \qed
\end{lemma}
We will fix $z$ and apply the Lemma for $g(A) \defeq P(A,z)$. 
Accordingly, we define $Q(A,z) \defeq \sum_{\udo B \supseteq A} P(B,z)$. Note that $Q(A,z)$ is the probability that bad coalescence in $\cR$ is triggered by 
a pair $(x,y) \in \cX^2$ with $D(x) \cap D(y) \supseteq A$ and $x|_{D(z)}^{} = y|_{D(z)}^{} = z$, prior to any other coalescence or mutation event. 

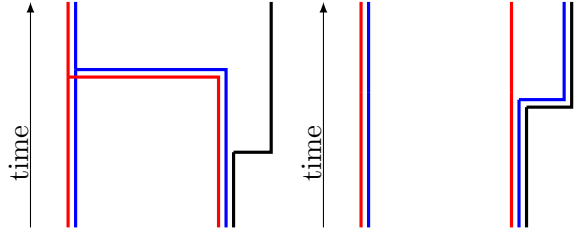
\begin{figure}[t]
\centering
\subfloat{
\begin{tikzpicture}
\draw[very thick, color=red] (0,0) -- (0,2); 
\draw[very thick, color=blue] (0.1,0) -- (0.1,2.1);
\draw[very thick, color=red, opacity = 0.4] (0,2) -- (0,3);
\draw[very thick, color=blue, opacity = 0.4] (0.1,2.1) -- (0.1,3);
\draw[very thick, color=red] (2,0) -- (2,2) -- (0,2);
\draw[very thick, color=blue] (2.1,0) -- (2.1,2.1) -- (0.1,2.1);
\draw[very thick, color=black] (2.2,0) -- (2.2,1);
\draw[very thick, color=black, opacity = 0.4] (2.2,1) -- (2.7,1) -- (2.7,3);
%Time arrow
\draw[-latex] (-0.5,0) -- (-0.5,3); 
\node[anchor=east, rotate=90] at (-0.66,1.5) {time};
\end{tikzpicture}
}
\subfloat{
\begin{tikzpicture}
\draw[very thick, color=red] (0,0) -- (0,1.8);
\draw[very thick, color=blue] (0.1,0) -- (0.1,1.8);
\draw[very thick, color=red, opacity=0.4] (0,1.8) -- (0,3);
\draw[very thick, color=blue, opacity=0.4] (0.1,1.8) -- (0.1,3);
\draw[very thick, color=red] (2,0) -- (2,1.8);
\draw[very thick, color=blue] (2.1,0) -- (2.1,1.7);
\draw[very thick, color=black] (2.2,0) -- (2.2,1.6);
\draw[very thick, color=blue, opacity=0.4] (2.1,1.7) -- (2.7,1.7) -- (2.7,3);
\draw[very thick, color=black, opacity=0.4] (2.2,1.6) -- (2.8,1.6) -- (2.8,3);
\draw[very thick, color=red, opacity=0.4] (2,1.8) -- (2,3);
\draw[-latex] (-0.5,0) -- (-0.5,3); 
\node[anchor=east, rotate=90] at (-0.66,1.5) {time};
\end{tikzpicture}

}
\caption{\label{fig:trackingtwoparticles}
Left: a coalescence of a pair of particles initially observed at $\{1,2\}$ and $\{1,2,3\}$, both supersets of $A = \{1,2\}$, triggering a bad coalescence in $\cR$.
Right: the split of the particle initially observed at $\{1,2,3\}$ makes this impossible. 
Sites $1$,$2$ and $3$ are color coded red, blue and black. Observed particles are opaque, while non-observed particles are transparent. 
}
\end{figure}

In order to compute $Q(A,z)$, we track the evolution (in $\cR$) of a fixed (ordered) pair of particles of types $y$ and $z$ as in the event defining $G(A)$; see Remark~\ref{rmk:asparticlesystems}. Note that this entails that both of them are observed at (some superset of) $A$. Each of the two particles suffers independent fragmentation due to recombination. Whenever possible, we keep tracking the unique particle that is still observed at a (now smaller) superset of $A$; see Fig.~\ref{fig:trackingtwoparticles}. Otherwise, if any of the two particles is hit by a recombination event that splits $A$, which happens at rate $2 \varrho \bar r_A^{}$ with 
\begin{equation*}
\bar r_A^{} \defeq \sum_{\substack{\cB \in \bP(S) \\ \cB|_A^{} \neq \{A\} }} r_\cB^{},
\end{equation*}
we terminate our observation. We also terminate the observation if any of the two particles is either hit by a mutation or coalesces  with a third particle. Note that such events occur with bounded rate. Clearly, the observed pair triggers the event defining $Q(A,z)$ if and only if the termination is due 
to coalescence of the pair, which happens at rate $1$ and hence with probability
\begin{equation*}
\frac{1}{2 \varrho \bar r_A^{} + 1 + \cO(1)} = \frac{1}{2 \varrho \bar r_A^{}} + \cO(\varrho^{-2}).
\end{equation*}
Now all we need to do is count the total number of such pairs, which is $2  \binom{\nu^{\supseteq A, D(z)}(z)}{2}$; keep in mind that we consider \emph{ordered} pairs, hence the factor of $2$. Therefore, 
\begin{equation*}
Q(A,z) = \frac{1}{\varrho \bar r_A^{}}\binom{\nu^{\supseteq A, D(z)}(z)}{2} + \cO(\varrho^{-2})
\end{equation*}
and Lemma~\ref{lem:exclincl} yields
\begin{equation} \label{F1x}
P(B,z)  = \sum_{\udo A \supseteq B} \frac{(-1)^{|A \setminus B|}}{\varrho \cdot \bar r_A^{}} \binom{\nu^{\supseteq A, D(z)}(z)}{2} + \cO(\varrho^{-2}).
\end{equation}
Inserting this into Eq.~\eqref{atleasttwosites} yields for $|D(z)| \geqslant 2$ that
\begin{equation*}
K(z) = \varrho^{-1} \sum_{D(z) \subseteq \udo A \subseteq S} \frac{(-1)^{|A \setminus D(z)|}}{\bar r_A^{}} \binom{\nu^{\supseteq A, D(z)}(z)}{2}  + \cO(\varrho^{-2}).
\end{equation*}
For $|D(z)| = 1$, we get from Eq.~\eqref{onesite}
\begin{equation*}
\begin{split}
K(z) &=  - \varrho^{-1} \sum_{D(z) \subsetneq \udo B \subseteq S} \sum_{B \subseteq \udo A \subseteq S}  \frac{(-1)^{|A \setminus B|}}{\bar r_A^{}} \binom{\nu^{\supseteq A, D(z)}(z)}{2}  + \cO(\varrho^{-2}) \\
&= \varrho^{-1} \sum_{D(z) \subsetneq \udo A \subseteq S}   \frac{(-1)^{|A \setminus D(z)|}}{\bar r_A^{}} \binom{\nu^{\supseteq A, D(z)}(z)}{2} 
\sum_{D(z) \subsetneq \udo B \subseteq A} (-1)^{|B|} + \cO(\varrho^{-2}) \\
&= \varrho^{-1} \sum_{D(z) \subsetneq \udo A \subseteq S}   \frac{(-1)^{|A \setminus D(z)|}}{\bar r_A^{}} \binom{\nu^{\supseteq A, D(z)}(z)}{2} + \cO(\varrho^{-2}),
\end{split}
\end{equation*}
where the last step follows from an elementary calculation using the binomial theorem. 
Finally, we have by Eq.~\eqref{zerosites}
\begin{equation*}
\begin{split}
K(\epsilon) &= \varrho^{-1} \sum_{\substack{\udo B \subseteq S \\ |B| \geqslant 2 }} \sum_{B \subseteq \udo A \subseteq S}  \frac{(-1)^{|A \setminus B|}}{\bar r_A^{}} \binom{\nu^{\supseteq A, D(\epsilon)}(\epsilon)}{2}  + \cO(\varrho^{-2}) \\
&= \varrho^{-1} \sum_{\substack{\udo A \subseteq S \\ |A| \geqslant 2}}  \frac{(-1)^{|A \setminus D(\epsilon)|}}{\bar r_A^{}} \binom{\nu^{\supseteq A, D(\epsilon)}(\epsilon)}{2}  \sum_{\substack{\udo B \subseteq S \\ |B| \geqslant 2 }} (|B| -1 )   + \cO(\varrho^{-2}) \\
&= \varrho^{-1} \sum_{\substack{\udo A \subseteq S \\ |A| \geqslant 2}}  \frac{(-1)^{|A \setminus D(\epsilon)|}}{\bar r_A^{}} \binom{\nu^{\supseteq A, D(\epsilon)}(\epsilon)}{2}  + \cO(\varrho^{-2}).
\end{split}
\end{equation*}
Again, the last equality is due to an elementary application of the binomial theorem. After exchanging $z \to x$, we see that the coefficients $K(x)$ 
in Eq.~\eqref{roughsamplingformula} are given exactly as in Eq.~\eqref{samplingformula}, which concludes the proof of Theorem~\ref{thm:maintheorem}.

\section*{Apppendix}
In the main text, we used the following elementary result about excursions in random walks
\begin{lemma}\label{lem:excursions}
Let $(X_t)_{t \geqslant 0}$ be a (continuous-time) random walk on $\ZZ$ with upward steps of bounded size and downward steps of size $1$. Assuming that the rates for the upward transitions stay constant and that the rate for a downward step scales linearly with $\varrho^{}$, we have that for any $C > 0$
\begin{equation*}
\PP( X \textnormal{ makes at least } k \textnormal{ upward steps before hitting } 0  \mid 0 < X_0 \leqslant C  ) = \cO(\varrho^{-k}),
\end{equation*}
where the implicit constant may depend on $C$.
\end{lemma}

\begin{proof}
Let $M$ be an upper bound for the size of the upward jumps. Assuming that $X$ has not hit $0$ immediately after the $k$-th upward step. Then, its maximal height at that time is (assuming that we \emph{only} see upward steps until then) $C + kM$. This means that in between any of the $k$ upward steps, there can not have been more than $C + kM$ downward steps.  Now, the distribution of the number $Y_i$ of downward moves between the $i-1$th and the $i$th upward move is given by i.i.d. geometric random variables with success probability
$1 - \cO(\varrho^{-1})$. Therefore, 
\begin{equation*}
\begin{split}
\PP( X \textnormal{ makes at least } & k \textnormal{ upward steps before hitting } 0  \mid 0 < X_0 \leqslant C  )  \\
& \leqslant \PP(\max_{1 \leqslant i \leqslant k} Y_i < C + kM) = (1 - (1 - \cO(\varrho^{-1}))^{C + kM})^k  = \cO(\varrho^{-k}).
\end{split}
\end{equation*}
Keep in mind that $C,k$ and $M$ are all constant.
\end{proof}

\section*{Acknowledgements}

It is a pleasure to thank Ellen Baake, Martina Favero, Paul Jenkins and Jere Koskela for stimulating discussions. Funded by the Deutsche Forschungsgemeinschaft (DFG, German Research Foundation) -- Project-ID 317210226 -- SFB 1283 and Project-ID 519713930.

\end{document}